\documentclass{amsart}
\usepackage{graphicx,pstricks}
\usepackage{amsmath, amsthm, amssymb}
\usepackage{wasysym}
\usepackage[x11names]{xcolor}
\usepackage[shortlabels]{enumitem}
\usepackage{tikz}
\usetikzlibrary{calc}
\usepackage{dsfont}
\usepackage{caption}
\usepackage[hidelinks]{hyperref}

\newcommand{\R}{\mathbb{R}}
\newcommand{\Z}{\mathbb{Z}}

\newcommand{\bP}{\mathbb{P}}

\definecolor{cb_blue}{RGB}{0,119,187}
\definecolor{cb_teal}{RGB}{0,153,136}
\definecolor{cb_magenta}{RGB}{238,51,119}
\definecolor{cb_orange}{RGB}{238, 119, 51}

\newtheorem{theorem}{Theorem}[section]
\newtheorem{lemma}{Lemma}[section]
\newtheorem{cor}{Corollary}[section]

\theoremstyle{definition}
\newtheorem{example}{Example}[section]
\newtheorem{definition}{Definition}[section]
\newtheorem{remark}{Remark}[section]

\title{Heat kernel estimates on book-like graphs}
\author{Emily Dautenhahn}
\thanks{Partially supported by NSF grant DMS-2054593.}
\address{Department of Mathematics and Statistics, Murray State University}
\author{Laurent Saloff-Coste}
\thanks{Partially supported by NSF grants DMS-2054593 and DMS-2343868.}
\address{Department of Mathematics, Cornell University}

\begin{document}

\begin{abstract}
In this paper, we prove two-sided heat kernel estimates on what we call ``book-like'' graphs. These are graphs consisting of pieces that satisfy the parabolic Harnack inequality that are glued together in a sufficiently nice way over a possibly infinite set of vertices. The prototypical example is gluing a copy of the square four-dimensional lattice $\Z^4,$ a copy of $\Z^5$, and a copy of $\Z^6$ by identifying their $x_1$-axes and taking the lazy simple random walk on this glued graph. Our results are flexible enough to handle perturbations of this example, for instance by adding diagonals to one of the lattices  or a few extra vertices/edges. 
\end{abstract}

\maketitle

\section{introduction}

The goal of this article is to obtain heat kernel estimates for a certain class of graphs which can be thought of as sufficiently nice pieces (``pages'') glued together in a sufficiently nice way via a gluing set (a ``spine''). The results we obtain can handle the case of gluing pieces satisfying the parabolic Harnack inequality via a finite set of vertices (the discrete version of some of the results of Grigor'yan and Saloff-Coste in \cite{lsc_ag_ends}), as well as gluing such graphs via an infinite set of vertices--under a specific set of hypotheses. The only existing work in a similar vein to this latter situation is work of Grigor'yan and Ishiwata \cite{ag_ishiwata}, which considers gluing copies of $\R^n$ via a paraboloid of revolution. While the hypotheses we make about the gluing set of vertices in this paper are fairly restrictive, they are different in flavor from those of \cite{ag_ishiwata}, and our results do not require precise symmetry. For example, results which apply to gluing lattices $\Z^d$ equally apply to gluing graphs that are lattice-like (e.g. quasi-isometric to lattices). The present paper can be seen as a culmination of our previous papers \cite{ed_lsc_trans, FK_glue}.

The prototypical example the reader should keep in mind throughout this paper is that of gluing lattices via a lower-dimensional lattice. For instance, as in Figure \ref{fig:lattice_ex} consider gluing a copy of $\Z^4$ (the four-dimensional square lattice), a copy of $\Z^5$, and a copy of $\Z^6$ by identifying their $x_1$-axes.

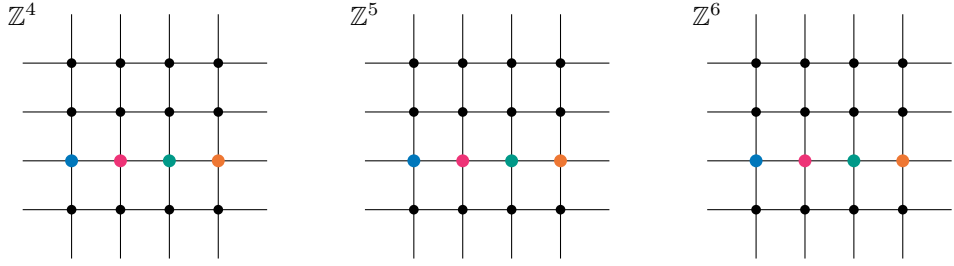
\begin{figure}[h]
\centering
\begin{tikzpicture}[scale=.65]
\foreach \x in {1,...,4} {
    \foreach \y in {1,...,4} {
         \node at (\x,\y) [circle, fill=black, inner sep=0.3ex] {};
        }
    }
    
    \foreach \x in {1,...,4} {
        \draw (\x,0)--(\x,5);
    }

        \foreach \y in {1,...,4} {
        \draw (0,\y)--(5,\y);
    }

 \foreach \x in {8,...,11} {
    \foreach \y in {1,...,4} {
         \node at (\x,\y) [circle, fill=black, inner sep=0.3ex] {};
        }
    }
    
        \foreach \x in {8,...,11} {
        \draw (\x,0)--(\x,5);
    }

        \foreach \y in {1,...,4} {
        \draw (7,\y)--(12,\y);
    }
    
     \foreach \x in {15,...,18} {
    \foreach \y in {1,...,4} {
         \node at (\x,\y) [circle, fill=black, inner sep=0.3ex] {};
        }
    }
    
        \foreach \x in {15,...,18} {
        \draw (\x,0)--(\x,5);
    }
    
        \foreach \y in {1,...,4} {
        \draw (14,\y)--(19,\y);
    }
    
      \node at (1,2) [circle, fill=cb_blue, inner sep = 0.4ex] {};
       \node at (8,2) [circle, fill=cb_blue, inner sep = 0.4ex] {};
        \node at (15,2) [circle, fill=cb_blue, inner sep = 0.4ex] {}; 
       \node at (2,2) [circle, fill=cb_magenta, inner sep = 0.4ex] {};
       \node at (9,2) [circle, fill=cb_magenta, inner sep = 0.4ex] {};
        \node at (16,2) [circle, fill=cb_magenta, inner sep = 0.4ex] {}; 
           \node at (3,2) [circle, fill=cb_teal, inner sep = 0.4ex] {};
       \node at (10,2) [circle, fill=cb_teal, inner sep = 0.4ex] {};
        \node at (17,2) [circle, fill=cb_teal, inner sep = 0.4ex] {}; 
        \node at (4,2) [circle, fill=cb_orange, inner sep = 0.4ex] {};
       \node at (11,2) [circle, fill=cb_orange, inner sep = 0.4ex] {};
        \node at (18,2) [circle, fill=cb_orange, inner sep = 0.4ex] {}; 
        \node at (0,5) {$\Z^4$};
        \node at (7,5) {$\Z^5$};
        \node at (14,5) {$\Z^6$};
\end{tikzpicture}
\captionsetup{justification=centering}
\caption{Adding extra dimensions to the figure as necessary, glue copies of $\Z^4, \Z^5, \Z^6$ by identifying their $x_1$-axes: all three vertices of each color shown are the same vertex.}\label{fig:lattice_ex}
\end{figure}

On this glued graph, take the lazy simple random walk (lazy SRW), which at each discrete time step stays in place with probability $1/2$ or moves to a neighboring vertex with equal probability. In any of the three lattices, away from the $x_1$-axis, the walk behaves like the lazy SRW on that lattice. At vertices on the $x_1$-axis, the walk has the chance to jump to neighboring vertices in any of the three lattices. The heat kernel $p(n,x,y)$ is the transition density of this random walk (with respect to the natural reference measure given at each vertex by the degree of this vertex). Let $D_x$ denote the dimension of the lattice to which vertex $x$ belongs, with $D_x = D_{\min} =\min\{4,5,6\} = 4$ if $x$ is on the $x_1$-axis. Let $|x|$ be the maximum of $1$ and the distance from $x$ to the $x_1$-axis, $d_{i_x}(x,y)$ be the distance between $x$ and $y$ staying in the lattice containing $x$ without visiting the $x_1$-axis with $d_{i_x}(x,y) = + \infty$ if $x, \ y$ are not in the same lattice, and $d_+(x,y)$ denote the distance between $x$ and $y$ where the $x_1$-axis must be visited. Then we show in Corollary \ref{lattice_lattice_hkest} that
\begin{align}
\begin{split}
p(n,x,y) \approx &\frac{C}{n^{\frac{D_x}{2}}}\exp \Big(-\frac{d_{i_x}^2(x,y)}{cn}\Big) 
+ \Bigg[ \frac{C}{n^{2} \, |x|^{D_x-3} \, |y|^{D_y-3}} \\[1ex]
&+ \frac{C}{n^{\frac{D_y}{2}} \, |x|^{D_x-3}}
+  \frac{C}{n^{\frac{D_x}{2}} \, |y|^{D_y-3}}  \Bigg] \exp\Big(-\frac{d_+^2(x,y)}{cn}\Big).\label{example_est}
\end{split}
\end{align}
Here the notation $\approx$ means we have matching upper and lower bounds of this form, where the constants $C, c$ differ in the upper and lower bound. These constants may depends on the dimensions of the lattices and how they are glued, but are independent of $n, x, y$. 

For instance, if both $x$ and $y$ are on the $x_1$-axis, then
\[ p(n,x,y) \approx \frac{C}{n^2} \exp\Big(-\frac{d^2(x,y)}{cn}\Big).\]
In fact, we already obtained such heat kernel estimates where both points belong to the gluing set in our preceding paper \cite{FK_glue}. However, in this paper, we are now able to give heat kernel estimates for any two points. 

If instead $x \in \Z^5, \ y \in \Z^6,$ then
\[ p(n,x,y) \approx \bigg[\frac{C}{n^2 |x|^2 |y|^3} + \frac{C}{n^3|x|^2} + \frac{C}{n^{5/2}|y|^3}\bigg]\exp\Big(-\frac{d^2(x,y)}{cn}\Big). \]
Note the smallest power of $n$, the $n^2$ appearing in the first term, is an effect from the smallest lattice (in terms of volume), the copy of $\Z^4$.

If $x, y \in \Z^5$, then
\begin{align*}
  p(n,x,y) &\approx \frac{C}{n^{5/2}}\exp\Big(-\frac{d^2(x,y)}{cn}\Big) + \frac{C}{n^2|x|^2|y|^2}\exp\Big(-\frac{d_+^2(x,y)}{cn}\Big).  
\end{align*}
Here the first term accounts for going between $x$ and $y$ while remaining in the copy of $\Z^5,$ and, indeed, this term is simply (comparable to) the Gaussian heat kernel estimate on $\Z^5$. The second term accounts for going between $x$ and $y$ while visiting the other lattices, in particular, while visiting the smallest lattice. Note that $d_+(x,y)$ is not necessarily equal to $|x|+|y|$ since it may be that the closest points on the $x_1$-axis to $x$ and $y$, respectively, are far away from each other. This is an effect that only appears when the gluing set is infinite. 

The reader can similarly use (\ref{example_est}) to see what happens for other locations of $x$ and $y$. Corollary \ref{lattice_lattice_hkest} gives a general formula that works for any number of lattices glued via a common lower-dimensional lattice. Constraints on dimension are due to requiring a notion of transience; note in our example with gluing over a line, the smallest possible lattice dimensions was $4$. One way this constraint appears in the formula is captured by $|x|^{D_x-3}$, whose exponent should be greater than $0$ because $p(n,x,y)$ cannot take values greater than $1$.

The approach of considering in a systematic way heat kernel estimates on spaces formed by gluing together (or cutting apart) other spaces first emerges in work of Grigor'yan and Saloff-Coste, who consider the setting of manifolds with ends in \cite{GS1, ag_lsc_extcptset,  ag_lsc_hittingprob, ag_lsc_harnackstability, lsc_ag_ends, ag_lsc_FKsurgery}. The main paper \cite{lsc_ag_ends} gives heat kernel estimates for the case that a finite number of manifolds satisfying the parabolic Harnack inequality are glued together over a compact set, provided at least one of the manifolds is non-parabolic (transient). Some particular special cases where all manifolds are parabolic (recurrent) are presented, but in general, this situation is complicated; see work of Grigor'yan, Ishiwata, and Saloff-Coste for further work in relaxing this hypothesis \cite{parabolic_ends, G_I_LSC_survey, Poincare_const}. The previous papers all consider the case of no boundary condition; the case of manifolds with ends with Dirichlet boundary condition or mixed Dirichlet and Neumann boundary condition was handled by the authors in \cite{ed_lsc_mixbdry}.

Another approach to this sort of gluing problem describes gluings that do not remain in the class of manifolds. For instance, consider gluing $\R^{d_1}$ and $\R^{d_2}$ (or $\R^{d_2}_+$) via a point, or via a compact set that is identified with a point. Since this gluing is not smooth and the spaces do not have the same topological dimension, such a space is not a manifold. The heat kernel on such spaces is defined using the theory of Dirichlet forms. This approach was initiated by Chen and Lou in \cite{chen_lou_BMvary}, where they considered a ``flag,'' meaning a copy of $\R_+$ glued to the origin of a copy of $\R^2$. In \cite{ooi_bmvd}, Ooi generalized this work to gluing $\R^{D_1}$ and $\R^{D_2}$ at the origin for any $D_1, D_2$. Further work of Lou obtains heat kernel estimates for Brownian motion with drift on the ``flag'' described above in \cite{lou_BMVD_drift}, and, in \cite{lou_li_distorted, lou_distorted}, Lou and Li and Lou give heat kernel estimates for distorted Brownian motion on a copy of $\R^3$ glued to a copy of $\R_+$ at the origin.

The above approaches all involved continuous time and continuous space. The present paper is concerned with such gluing problems in discrete time and discrete space, in particular, countably infinite graphs. Our previous papers \cite{ed_lsc_trans, FK_glue} build up the machinery needed to obtain the results in this paper, some of which may also be found in the first author's PhD thesis \cite{Emily_thesis}. In this paper, we give full two-sided matching heat kernel estimates in the setting of gluing a finite number of countably infinite graphs satisfying the parabolic Harnack inequality in such a way that the resulting graph is ``book-like''. We require the graphs we glue to be uniformly $S$-transient, a notion we introduced in \cite{ed_lsc_trans}. That paper also proves hitting probability estimates we make use of here. The results of \cite{FK_glue} give gluing set to gluing set heat kernel estimates, which are expanded upon here to get the full estimates. Our main result, Theorem \ref{main_general_thm} provides the most general heat kernel estimates on book-like graphs. Specific examples of this theorem show how to use it in practice, as well as its limitations. Corollary \ref{graphs_finite_case} gives the significantly simpler estimates that occur in the case the gluing set is finite; this is a discrete time and space version of Theorems 4.9 and 5.10 in \cite{lsc_ag_ends}. While we do not state a general theorem that can handle the case where only one piece need be uniformly $S$-transient as in Theorem 6.6 of \cite{lsc_ag_ends}, Examples \ref{Z_D_tail} and \ref{Z_D_plane} are of mixed transience and recurrent type. 

Our key assumption on the gluing set is that the graph be ``book-like''. A graph is ``book-like'' if all points in the gluing spine always see all pages, like the spine of a book. For instance, in the example considered earlier in this introduction, all points on the $x_1$-axis can always see all three of the other lattices. The simplest sort of example of a non-book-like graph is as follows: Again consider a copy $\Z^4$, one of $\Z^5,$ and one of $\Z^6$. However, now identify the $x_1$-axes in the $\Z^4$ and the $\Z^5$ and the $x_2$-axes in the $\Z^5$ and the $\Z^6$. The gluing spine is now a ``cross'', and points on the spine can only easily see two of the three pages. For instance, points on the $x_2$-axis increasingly far away from the origin are increasingly far away from the $\Z^4$ page. Example 7.5 of \cite{FK_glue} discussed some of the difficulty with this cross example. The authors hope to be able to use similar ideas and methods to handle that example in a future paper; however, we also believe this is the most complicated type of example we will be able to compute. 

We remark here upon the relationship between results in the discrete and continuous cases. Ideally, the authors would like to use the discrete approach presented here as a blueprint for obtaining results complementary to those of Grigor'yan and Ishiwata \cite{ag_ishiwata} that allow for gluing over non-compact sets that do not have strict symmetry in the continuous setting of manifolds. Some obstacles to this approach include an appropriate continuous setting analog of the hitting probability results of \cite{ed_lsc_trans} as well as the difficulty in describing gluings in the manifolds setting. It is worth recalling that there is no general theorem that says heat kernel estimates automatically translate between any related discrete and continuous spaces. (For instance, consider a manifold and an associated $\varepsilon$-net as continuous and discrete versions of the same space.) However, certain functional inequalities may transfer over quasi-isometric spaces, and some combinations of functional inequalities are equivalent to particular heat kernel estimates. For instance, Coulhon and Saloff-Coste show the parabolic Harnack inequality (and therefore two-sided Gaussian heat kernel estimates) transfers over quasi-isometry in \cite{TC_LSC_IsoInfini}. Indeed, we made use of this idea of transferring functional inequalities (in particular, relative Faber-Krahn inequalities) between quasi-isometric spaces in \cite{FK_glue}. However, in the present setting, the heat kernel estimates we obtain are sufficiently complicated that trying to equate them to functional inequalities or prove a theorem enabling their transfer to other settings seems quite challenging. Therefore, we instead hope that the overall method of this paper (and those it builds upon) can still be applied in the continuous setting, with care taken to address specific difficulties that arise. 

The rest of this paper is organized as follows. Section \ref{notation} sets out the necessary notation, hypotheses, and particular construction of cutting and gluing graphs we consider. Section \ref{glue_setup} describes the general framework for estimating the heat kernel in terms of a gluing formula. Terms appearing in these gluing formulas can be estimated using results from our papers \cite{ed_lsc_trans, FK_glue}. Section \ref{abstract_est} applies the gluing formulas to obtain general heat kernel estimates for the kind of book-like graph we consider here, with the main result being Theorem \ref{main_general_thm}. The last three sections of this chapter apply Theorem \ref{main_general_thm} to several examples where it is possible to obtain more concrete estimates. Section \ref{finiteglue} addresses the case of a finite gluing set, and Section \ref{proto_ex} covers the main motivating example of gluing lattices of varying dimensions along a smaller dimensional lattice. Section \ref{morebooklike_ex} gives additional examples of gluing lattices along a half-space or along a cone. It also contains two examples where one lattice is recurrent; these examples use the $h$-transform technique. 

\section{Notation and Set-up}\label{notation}

The notation and constructions used in this paper match those found in \cite{ed_lsc_trans, FK_glue}. For the sake of space,  here we make our definitions fairly brief. Additional context for random walk structures on graphs and $S$-transience may be found in Sections 1 and 2 of \cite{ed_lsc_trans}, and a more in-depth description of our gluing/cutting construction is given in Section 3 of \cite{FK_glue}. The overall goal of this section is to introduce the concept of a book-like graph and to state the main hypotheses we will make in the rest of the paper. 

\subsection{Random Walk Structure on Graphs}

A graph $\Gamma$ is a set of vertices $V$ with a set of edges $E,$ where $E$ is a subset of the set of all pairs of elements of $V$. (Such graphs are undirected.) In general, we consider all graphs to be simple and connected unless stated otherwise. 

Any such graph $\Gamma$ has a distance metric $d=d_\Gamma$ defined by taking the shortest path of edges between any two vertices. The notation $x \sim y$ means $\{x,y\} \in E$; we may also say that $x,y$ are neighbors. The notation $x \simeq y$ means $x \sim y$  or $x=y.$ 

We are interested in studying graphs with random walk structures. We will define such random walk structures as triples $(\Gamma, \pi, \mu)$ or $(\Gamma, \mathcal{K}, \pi),$ where $\pi$ is a vertex weight function (measure), $\mu$ is an edge weight function (conductance), and $\mathcal{K}$ is a Markov kernel. 

\begin{definition}[Edge and Vertex Weights]
    Any function $\pi : V \to \R_{>0}$ is a vertex weight. Given such a weight $\pi,$ a function $\mu: E \to \R_{\geq 0}$ is an edge weight if it satisfies the following three properties:
    \begin{itemize}
        \item Adapted to the edges: $\mu_{xy} \not= 0 \iff {x,y} \in E$
        \item Symmetric: $\mu_{xy} = \mu_{yx}$
        \item Subordinate to the vertex measure: $\sum_{y\sim x} \mu_{xy} \leq \pi(x)$ for all $x \in V$.
    \end{itemize}
\end{definition}

We use $B(x,r) = \{ y \in \Gamma : d_\Gamma(x,y) \leq r\}$ to denote the closed ball of radius $r$ centered at $x$ in the graph. The notation $V(x,r)$ denotes the volume of $B(x,r)$ with respect to $\pi,$ that is, $V(x,r) = \sum_{y \in B(x,r)} \pi (y).$ As seen in this paragraph, we often abuse notation by writing $\Gamma$, the graph, and $V$, the set of vertices of $\Gamma$, interchangeably.

\begin{definition}[Markov kernel on a graph]\label{MarkovKernel} 
    Given a vertex weight $\pi$ and an edge weight $\mu$, we define a Markov kernel $\mathcal{K}$ on $\Gamma$ via 
    \begin{equation}\label{MK_formula}
\mathcal{K}(x,y) =\begin{cases} \frac{\mu_{xy}}{\pi(x)}, & x \not = y \\ 1 - \sum_{z \sim x} \frac{\mu_{xz}}{\pi(x)} , & x=y. \end{cases}
\end{equation}
\end{definition}

An important distinction is that while $\Gamma$ is not allowed to have loops as a graph, the Markov kernel $\mathcal{K}$ may nonetheless stay in place.

In Definition \ref{MarkovKernel}, $\mathcal{K}$ was defined given $\pi$ and $\mu$; however, given $\mathcal{K}$ and $\pi$, equation (\ref{MK_formula}) uniquely defines $\mu$. 

The Markov kernel $\mathcal{K}$ is \emph{reversible} with respect to the vertex measure $\pi$:
\[ \mathcal{K}(x,y) \pi(x) = \mathcal{K}(y,x) \pi(y) \quad \forall x, y \in \Gamma.\]
One reason the above line is important is that while the graph $\Gamma$ is undirected, $\mathcal{K}(x,y)$ need not be equal to $\mathcal{K}(y,x)$ (refer back to (\ref{MK_formula})). 

\begin{definition}[Heat kernel on a graph]
Let $\mathcal{K}^n(x,y)$ denote the $n$-th convolution power of $\mathcal{K}(x,y)$ and $(X_n)_{n \geq 0}$ denote a random walk on $\Gamma$ driven by $\mathcal{K}$. Then $\mathcal{K}^n$ satisfies $\mathcal{K}^n(x,y) = \bP^x(X_n = y).$ The quantity $\mathcal{K}^n(x,y)$ need not be symmetric. However, the \emph{transition density} of $\mathcal{K}$,
\[ p(n,x,y) = p_n(x,y) = \frac{\mathcal{K}^n(x,y)}{\pi(y)}, \]
is always symmetric. The transition density $p(n,x,y)$ is also the \emph{heat kernel} of $(\Gamma, \pi, \mu)$.
\end{definition}

In this paper, we make the following two assumptions regarding the weights on the graph. 

\begin{definition}[Controlled weights]\label{cont_weights}
A graph $\Gamma$ has \emph{controlled weights} if there exists a constant $C_c > 1$ such that 
\begin{equation}\label{controlled} \frac{\mu_{xy}}{\pi(x)} \geq \frac{1}{C_c} \quad \forall x \in \Gamma, \ y \sim x .\end{equation}
\end{definition}

This assumption implies that $\Gamma$ is locally uniformly finite (that is, there is a uniform bound on the degree of any vertex) and that for $x \sim y,$ we have $\mu_{xy} \approx \pi(x) \approx \pi(y)$, where the symbol $\approx$ means comparable up to constants that do not depend on $x,y$, as an abuse of Definition \ref{approx_def} below. 

\begin{definition}[The notation $\approx$]\label{approx_def}
For two functions of a variable $x,$ the notation $f \approx g$ means there exist constants $c_1, c_2$ (independent of $x$) such that 
\[ c_1 f(x) \leq g(x) \leq c_2 f(x).\]
\end{definition}

\begin{definition}[Uniformly lazy]\label{unif_lazy}
A weight pair $(\mu, \pi)$ is \emph{uniformly lazy} if there exists $C_e \in (0,1)$ such that 
\[ \sum_{y \sim x} \mu_{xy} \leq (1-C_e) \pi(x) \quad \forall x \in V, \ y \sim x.\]
A Markov kernel $\mathcal{K}$ is \emph{uniformly lazy} if there exists $C_e \in (0,1)$ such that 
\[ \mathcal{K}(x,x) = 1 - \sum_{z \sim x} \frac{\mu_{xz}}{\pi(x)} \geq  C_e \quad \forall x \in \Gamma. \]
The two conditions above are equivalent. In this case, the Markov chain is \emph{aperiodic}. 
\end{definition}

\emph{Unless stated otherwise, we will consider all random walk structures appearing to be uniformly lazy and have controlled weights.}

\begin{definition}[Lazy simple random walk]\label{lazy_SRW}
Let $\Gamma$ be a locally uniformly finite graph. Define weights $\pi(x) = 2\text{deg}(x)$ for all $x \in \Gamma$ and  $\mu_{xy} = 1$ for all ${x,y} \in E$, where $\text{deg}(x)$ returns the degree of vertex $x$. At each time step, the lazy simple random walk (lazy SRW) stays in place with probability $1/2$ and otherwise moves to a neighbor of the current vertex with equal probability.  
\end{definition}

The lazy simple random walk has controlled weights and is uniformly lazy, so it is always possible to impose a random walk structure of the type we consider on any graph with uniformly bounded vertex degree.

\begin{remark}
The controlled weights hypothesis is very important to our work. The uniformly lazy hypothesis is not strictly necessary: We have made it out of convenience to avoid parity problems. 
\end{remark}

\subsection{Random walks on subgraphs}\label{subgraph}

Given a graph $\widehat{\Gamma} =(\widehat{V}, \widehat{E}),$ for any subset $V$ of $\widehat{V}$ we can construct a graph $\Gamma$ with vertex set $V$ and edge set $E$ where $\{ x , y \} \in E$ if and only if $x,y \in V$ and $\{x,y\} \in \widehat{E}.$ As usual, we abuse notation and use the same symbol $\Gamma$ to denote both a subset of $\widehat{V}$ and its associated subgraph. 

Any such subgraph has a distance function $d_\Gamma$ that satisfies $d_{\widehat{\Gamma}}(x,y) \leq d_\Gamma(x,y)$ for all $x, y \in \Gamma.$ A subgraph $\Gamma$ also inherits a random walk structure from $(\widehat{\Gamma}, \pi, \mu)$ by setting $\pi_{\Gamma}(x) = \pi_{\widehat{\Gamma}}(x)$ and $\mu^{\Gamma}_{xy} = \mu^{\widehat{\Gamma}}_{xy}$ for all $x,y \in V, \ \{x,y\} \in E.$ (Hence we may simply use $\pi, \mu$ without indicating the precise graph, provided $x,y \in \Gamma.$) 

There are two natural (sub)Markov kernels on $\Gamma.$

\begin{definition}[Neumann Markov kernel on a subgraph]
Defining a Markov kernel on $\Gamma$ as in Definition \ref{MarkovKernel} yields the Neumann Markov kernel of $\Gamma$ (with respect to $\widehat{\Gamma}$), which we denote $\mathcal{K}_{\Gamma, N}:$
\begin{equation}
\mathcal{K}_{\Gamma,N}(x,y) = \begin{cases} \frac{\mu^\Gamma_{xy}}{\pi(x)} = \frac{\mu_{xy}^{\widehat{\Gamma}}}{\pi(x)}, & x \not = y, \ x,y \in V \\ 1 - \sum_{z \sim x} \frac{\mu^{\Gamma}_{xz}}{\pi(x)} = 1 - \sum_{z \sim x, \, z \in V} \frac{\mu^{\widehat{\Gamma}}_{xz}}{\pi(x)}, & x=y \in V.\end{cases}
\end{equation}

\end{definition}

\begin{definition}[Dirichlet Markov kernel on a subgraph]
The Dirichlet Markov kernel of $\Gamma$ (with respect to $\widehat{\Gamma}$), denoted $\mathcal{K}_{\Gamma, D}(x,y)$, is 
\begin{equation}
\mathcal{K}_{\Gamma,D}(x,y) = \mathcal{K}_{\widehat{\Gamma}}(x,y) \mathds{1}_{V}(x) \mathds{1}_{V}(y) = \begin{cases} \frac{\mu^\Gamma_{xy}}{\pi(x)}, & x \not = y, \ x, y \in V \\ 1 - \sum_{z \sim x, z \in \widehat{V}} \frac{\mu^{\widehat{\Gamma}}_{xz}}{\pi(x)} , & x=y \in V,\end{cases}
\end{equation}
where $\mathds{1}_V(x) = 1$ if $x \in V$ and zero otherwise. When $V \not = \widehat{V}, \ \mathcal{K}_{\Gamma, D}$ is only a sub-Markovian kernel. 
\end{definition}

There are also two natural notions for the boundary of $\Gamma.$ 

\begin{definition}[Boundary of a subgraph]
The \emph{(exterior) boundary} of $\Gamma$ is the set of points that do not belong to $\Gamma$ with neighbors in $\Gamma,$
\[ \partial \Gamma = \{ y \in \widehat{\Gamma} \setminus \Gamma: \exists x \in \Gamma \text{ s.t. } d_{\widehat{\Gamma}}(x,y) = 1\}.\]
By contrast, the \emph{inner boundary} of $\Gamma$ is the set of points inside $\Gamma$ with neighbors outside of $\Gamma,$ 
\[ \partial_I \Gamma = \{ x \in \Gamma: \exists y \not \in \Gamma \text{ s.t. } d_{\widehat{\Gamma}}(x,y) = 1\}.\]
\end{definition}

\subsection{Properties of Graphs}

These first two definitions are geometric properties of subgraphs. The only difference in these definitions comes in which distance function is used in property (a).

\begin{definition}[Uniform]\label{unif}
A subgraph $\Gamma$ of a graph $\widehat{\Gamma}$ is \emph{uniform} in $\widehat{\Gamma}$  if there exist constants $0<c_u, \, C_U < +\infty$ such that for any $x,y \in \Gamma$ there is a path $\gamma_{xy} = (x_0 = x, x_1, \dots, x_k =y)$ between $x$ and $y$ in $\Gamma$ such that 
\begin{enumerate}[(a)]
\item $k \leq C_U d_{\widehat{\Gamma}}(x,y) $
\item For any $j \in \{0, \dots, k\},$
\[ d_{\widehat{\Gamma}}(x_j, \partial \Gamma) = d_{\widehat{\Gamma}}(x_j, \widehat{\Gamma} \setminus \Gamma) \geq c_u (1 + \min \{ j, k-j\}).\]
\end{enumerate}
\end{definition}

\begin{definition}[Inner uniform]\label{innerunif}
A subgraph $\Gamma$ of $\widehat{\Gamma}$ is \emph{inner uniform} in $\widehat{\Gamma}$  if there exist constants $0<c_u, \, C_U < +\infty$ such that for any $x,y \in \Gamma$ there is a path $\gamma_{xy} = (x_0 = x, x_1, \dots, x_k =y)$ between $x$ and $y$ in $\Gamma$ such that 
\begin{enumerate}[(a)]
\item $k \leq C_U d_{\Gamma}(x,y) $
\item For any $j \in \{0, \dots, k\},$
\[ d_{\widehat{\Gamma}}(x_j, \partial \Gamma) = d_{\widehat{\Gamma}}(x_j, \widehat{\Gamma} \setminus \Gamma) \geq c_u (1 + \min \{ j, k-j\}).\]
\end{enumerate}
\end{definition}

For more details on (inner) uniform graphs in the relevant setting and for further literature references, please see Section 4 of \cite{FK_glue}.

We now introduce several definitions and theorems needed to eventually define a Harnack graph. 

\begin{definition}[Doubling]
A graph is said to be \emph{doubling} if there exists a constant $D$ such that for all $r>0, \ x \in \Gamma,$ 
\begin{equation}
V(x,2r) \leq D V(x,r).
\end{equation}
\end{definition}

\begin{definition}[Poincar\'{e} inequality]
We say that $\Gamma$ satisfies the \emph{Poincar\'{e} inequality} if there exist constants $C_p >0, \ \kappa \geq 1$ such that for all $r > 0,\ x \in \Gamma,$ and functions $f$ supported in $B(x,\kappa r),$ 
\begin{equation}
\sum_{y \in B(x,r)} |f(y) - f_B|^2 \, \pi(y) \leq C_p \ r^2 \sum_{y,z \in B(x, \kappa r)} |f(y)-f(z)|^2 \, \mu_{yz},
\end{equation}
where $f_B$ is the (weighted) average of $f$ over the ball $B=B(x,r),$ that is,
\[ f_B = \frac{1}{V(x,r)} \sum_{y \in B(x,r)} f(y)\, \pi(y).\]
\end{definition}

Under doubling, the Poincar\'{e} inequality with constant $\kappa > 1$ (which appears in $B(x, \kappa r)$ on the right hand side) is equivalent to the Poincar\'{e} inequality with $\kappa = 1$ (see Corollary A.51 of \cite{Barlow_graphs}).  

\begin{definition}[Harmonic function]
A function $h: \Gamma \to \R$ is \emph{harmonic} (with respect to $\mathcal{K}$) if 
\begin{equation}\label{harmonic} h(x)  = \sum_{y \in \Gamma} \mathcal{K}(x,y) h(y) \quad \forall x \in \Gamma.\end{equation}
Given a subset $\Omega$ of $\Gamma$ (usually a ball), $h$ is harmonic on that set if (\ref{harmonic}) holds for all points in $\Omega$; this requires that $h$ be defined on $\{v\in \Gamma: \exists \omega\in \Omega,\ v\simeq \omega\} = \Omega \cup \partial \Omega$. 
\end{definition}
As $\mathcal{K}(x,y) = 0$ unless $y \simeq x,$ the sum over $y \in \Gamma$ in (\ref{harmonic}) can be replaced by a sum over $y \simeq x.$ 

\begin{definition}[Elliptic Harnack inequality]
A graph $(\Gamma, \pi, \mu)$ satisfies the \emph{elliptic Harnack inequality} if there exist $\eta \in (0,1),\ C_H>0$ such that for all $r>0, \ x_0 \in \Gamma,$ and all non-negative harmonic functions $h$ on $B(x_0, r),$ we have 
\[ \sup_{B(x_0,\, \eta r)} h \leq C_H \inf_{B(x_0,\, \eta r)} h. \]

We say a graph satisfies the elliptic Harnack inequality \emph{up to scale} $R>0$ if the above definition holds for $0<r<R,$ where the constant $C_H$ depends upon $R.$ 
\end{definition}

\begin{definition}[Solution of discrete heat equation]
A function $u: \Z_{+} \times \Gamma \to \R$ \emph{solves the discrete heat equation} if
\begin{equation}\label{heat_eqn} u(n+1, x) - u(n,x) = \sum_{y \in \Gamma} \mathcal{K}(x,y)[ u(n,y) - u(n,x)]  \quad \forall n \geq 1, \ x \in \Gamma.\end{equation}
Given a discrete space-time cylinder $Q = I \times B,$ $u$ solves the heat equation on $Q$ if (\ref{heat_eqn}) holds there (this requires that for each $n\in I$, $u(n,\cdot)$ is defined on $\{z\in \Gamma: \exists x \in B,\ z\simeq x\} = B \cup \partial B$).\end{definition} 

\begin{definition}[Parabolic Harnack inequality]
A graph $(\Gamma, \pi, \mu)$ satisfies the (discrete time and space) \emph{parabolic Harnack inequality} if: there exist $\eta \in (0,1), \ 0 < \theta_1 < \theta_2 < \theta_3 < \theta_4$ and $C_{PH}>0$ such that for all $s,r >0, \ x_0 \in \Gamma,$ and every non-negative solution $u$ of the heat equation in the cylinder $Q = (\Z_{+} \cap [s, s+ \theta_4 r^2]) \times B(x_0,r),$ we have
\[ u(n_-, x_-) \leq C_{PH} \, u(n_+, x_+) \]
for all $(n_-, x_-) \in Q_-, \ (n_+, x_+) \in Q_+$ such that $d(x_-, x_+) \leq n_+ - n_- ,$ where 
\begin{align*}
&Q_- = (\Z_+ \cap [ s + \theta_1 r^2, s + \theta_2 r^2]) \times B(x_0, \eta r)\\ 
&Q_+ = (\Z_+ \cap [ s+ \theta_3 r^2, s+ \theta_4 r^2]) \times B(x_0, \eta r).
\end{align*}

We say a graph satisfies the parabolic Harnack inequality \emph{up to scale} $R>0$ if the above definition holds for $0<r<R,$ where the constant $C_{PH}$ now depends on $R$.
\end{definition}

The parabolic Harnack inequality obviously implies the elliptic version. Any graph with controlled weights satisfies the parabolic Harnack inequality at scale $1$, and therefore satisfies the parabolic Harnack inequality up to any finite scale $R$. 

As first shown for manifolds by Saloff-Coste \cite{PSHDuke} and Grigor'yan \cite{Gri}, there are several equivalent characterizations of the parabolic Harnack inequality. The discrete version of this theorem given below is due to Delmotte \cite{Delmotte_PHI}. 

\begin{theorem}[Theorem 1.7 in \cite{Delmotte_PHI}]\label{VD_PI}
Given $(\Gamma, \pi, \mu)$  (or $(\Gamma, \mathcal{K}, \pi)$) where $\Gamma$ has controlled weights and $\mathcal K$ is uniformly lazy, the following are equivalent: 
\begin{enumerate}
\item[(a)] $\Gamma$ is  doubling and satisfies the Poincar\'{e} inequality 
\item[(b)] $\Gamma$ satisfies the parabolic Harnack inequality
\item[(c)] $\Gamma$ satisfies two-sided Gaussian heat kernel estimates, that is there exists constants $c_1, c_2, c_3, c_4 >0$ such that for all $x,y \in \Gamma, \ n \geq d(x,y),$
\begin{equation}\label{2sided_graphs}
\frac{c_1 }{V(x, \sqrt{n})} \exp\Big(-\frac{d^2(x,y)}{c_3 n}\Big) \leq p(n,x,y) \leq \frac{c_3}{V(x, \sqrt{n})} \exp\Big(-\frac{d^2(x,y)}{c_4 n}\Big).
\end{equation}
\end{enumerate}
\end{theorem}

\begin{definition}[Harnack graph]\label{Harnack_graph}
If $(\Gamma, \pi, \mu)$ satisfies any of the conditions in Theorem \ref{VD_PI}, we call $\Gamma$ a \emph{Harnack graph}. 
\end{definition}

\begin{definition}[Abuse of $\approx$]\label{approx_def2}
Later in the paper, we will often write expressions such as (\ref{2sided_graphs}) in the form
\[ p(n,x,y) \approx \frac{C}{V(x,\sqrt{n})} \exp\Big(-\frac{d^2(x,y)}{cn}\Big).\]
Such uses of $\approx$ should be taken to mean there are matching upper and lower bounds with different values of the constants $c, C$ in the upper and lower bounds. In the event there are multiple lines in a row, the constants may change from line to line, and if there are multiple terms added together in an estimate, $c, C$ can be thought of as different or as the same (by taking maximums/minimums of the differing constants). These constants do not depend on $x, y, n$.
\end{definition}

\begin{definition}[Hitting times and probabilities on graphs]\label{hitprob_def}
Let $K$ be a subgraph of $(\Gamma, \pi, \mu)$. We define the following hitting times and hitting probabilities, where $v \in K$:
\begin{itemize}
\item $\tau_K := \min\{ n \geq 0: X_n \in K\}$, the (first) hitting time of $K$
\item $\psi_K(x) := \bP^x(\tau_K  < + \infty),$ the chance of hitting $K$ in finite time, given the walk starts at $x$ 
\item $\psi_K(x,v) := \bP^x (X_{\tau_K} = v, \, \tau_K < + \infty),$ the chance of hitting $K$ for the first time at the point $v,$ given the walk starts at $x$
\item $\psi_K(n,x,v) := \bP^x (X_{\tau_K} =v, \, \tau_K \leq n),$ the chance of hitting $K$ for the first time at the point $v,$ and doing so in time less than or equal to $n$ 
\item $\psi_K'(n,x,v) := \psi_K(n, x,v) - \psi_K(n-1,x,v) = \bP^x (X_{\tau_K} = v,\, \tau_K =n),$ the chance of hitting $K$ for the first time at $v$ at the time $n$.
\end{itemize}
\end{definition}

The following notion of transience was defined in \cite{ed_lsc_trans}; see Section 2 of that paper for examples. 

\begin{definition}[Uniform $S$-transience]\label{S_trans_def}
Let $(\widehat{\Gamma}, \mathcal{K}, \pi)$ be a connected graph with controlled weights and $K$ be a subset of $\widehat{\Gamma}$ such that $\Gamma := \widehat{\Gamma} \setminus K$ is connected. We say the subgraph $\Gamma$ is \emph{$S$-transient} (``survival"-transient) or that the graph \emph{$\widehat{\Gamma}$ is $S$-transient with respect to the set $K$} if there exists $x \in \widehat{\Gamma}$ such that $\psi_K(x) < 1.$ (If this is not the case, $\Gamma$ is said to be \emph{$S$-recurrent}.)

Further, $\Gamma$ is \emph{uniformly $S$-transient}, ($\widehat{\Gamma}$ is \emph{uniformly $S$-transient with respect to} $K$), if there exist $L, \varepsilon >0$ such that $d(x,K) \geq L$ implies that $\psi_K(x) \leq 1- \varepsilon.$ 
\end{definition}

Uniform $S$-transience is a useful hypothesis because we proved in \cite{ed_lsc_trans} that it implies the Dirichlet and Neumann heat kernels of a subgraph are comparable (when appropriate other hypotheses also hold). 

\begin{lemma}[Corollary 3.14 of \cite{ed_lsc_trans}]\label{D_approx_N}
Assume that $\widehat{\Gamma}$ is a Harnack graph that is uniformly  $S$-transient with respect to $K$ and that $\Gamma := \widehat{\Gamma} \setminus K$ is inner uniform. Then there exist constants $0<c, \, C < +\infty$ such that 
\begin{align}\label{Dirichlet_approx_Neumman}
cp_{\Gamma,N}(Cn, x, y) \leq p_{\Gamma,D}(n,x,y) \leq p_{\Gamma,N}(n,x,y). 
\end{align}
\end{lemma}

\subsection{Book-like Graphs}

The cutting/gluing construction below is described in detail with plentiful examples in Section 3 of \cite{FK_glue}. Therefore, here we give a straightforward description. 

\subsubsection{Gluing graphs}

Begin with a finite number of ``pages'' (sometimes called ``pieces''), graphs $(\Gamma_i, \pi_i, \mu^i)$ for $i=1,\dots, l$. The pages play the role of ``ends'' in the manifolds with ends setting. The pages should be thought of as countably infinite connected graphs satisfying our main hypotheses of controlled and uniformly lazy weights. 

Additionally, assume we have gluing ``spine'' (sometimes called a ``gluing set''), $(\Gamma_0, \pi_0, \mu^0).$ This corresponds to the central compact set in the manifolds with ends case. \textbf{We allow for $\Gamma_0$ to be disconnected}. We still require that weights on $\Gamma_0$ be adapted, subordinate, controlled, and uniformly lazy, though we now allow the weight of totally disconnected vertices to be zero. 

Glue the pages $\Gamma_i, \ 1 \leq i \leq l$ along the spine $\Gamma_0$ as follows: Assume each $\Gamma_i$ comes with a marked set of vertices (a ``margin''); label this set of vertices $G_i.$ Further, assume that for each $i \in \{1, \dots, l\}$ the spine $\Gamma_0$ contains a set of vertices $G_i'$ that is a copy of $G_i.$ (Note the $G_i'$ need not be disjoint.) For each $i,$ we assume there is an identification map (bijection) between $G_i$ and $G_i'$. Moreover, this identification map should satisfy the compatibility condition that there exist constants $c_I, C_I$ such that if $x \in G_i, x' \in G_i'$ are identified, then 
\begin{align}\label{compatible_weights} c_I \pi_0(x') \leq \pi_i(x) \leq C_I \pi_0(x').\end{align}

Equation (\ref{compatible_weights}) ensures the weights at vertices glued together are compatible, with uniform constants applying to the whole graph. This condition implies the weights on any glued edges are also comparable. The graph resulting from gluing the pages to the spine according to these identification maps is called $\Gamma.$ (See Remark 3.1 of \cite{FK_glue} for the appropriate slight modification of (\ref{compatible_weights}) in the situation where $x \in \Gamma_0$ has no neighbors in $\Gamma_0$ due to $\Gamma_0$ being disconnected.)

In this paper, we consider identification as literally identifying the vertices (as opposed to joining the vertices with an edge). It is therefore sensible to define a random walk structure on the glued graph $\Gamma$ by defining the weights
\[ \pi(x) = \sum_{i=0}^l \pi_i(x) \quad \text{ and } \quad \mu_{xy} = \sum_{i=0}^l \mu_{xy}^i \quad \forall x, y \in V,\]
where we take the convention that $\pi_i(x) = 0$ if $x \not \in V_i$ (the vertex set of $\Gamma_i$) and $\mu_{xy}^i =0$ if $\{ x,y \} \not \in E_i$ (the edge set of $\Gamma_i$). 

We ask that the following be true of the glued graph $\Gamma$:
\begin{enumerate}[A.]
\item $\Gamma$ is connected.
\item There exists a number $\alpha>0$ such that $\Gamma_0,$ seen as a subgraph of $\Gamma,$ is $\alpha$-connected, that is, $[\Gamma_0]_\alpha,$ the $\alpha$-neighborhood of $\Gamma_0$ in $\Gamma,$ is connected.

\item When seen as a subgraph of $\Gamma,$ for all $i=1,\dots, l,$ we have $\partial_I \Gamma_i = G_i.$

\item The identification between $G_i, \ G_i'$ satisfies the description given above, i.e. is a bijection between vertices with compatible weights. 
\end{enumerate}

Under these conditions, $(\Gamma, \pi, \mu)$ is a simple connected graph with controlled and uniformly lazy weights (Lemma 3.2 of \cite{FK_glue}). 

\begin{remark}\label{alpha_conn}
    In fact, the assumption/hypothesis B. above that the spine $\Gamma_0$ be $\alpha$-connected is not necessary. We have stated this hypothesis here to match with the construction given in our previous paper \cite{FK_glue}. However, the attentive reader can verify that this hypothesis was not used anywhere in that paper, nor is it necessary for any of the arguments we develop here. Examples \ref{lattice_2arb} and \ref{lattice_sparse} below fall in the class of spines that are not $\alpha$-connected. In all other examples presented in the paper, the spine is $\alpha$-connected, and this fact allows us to simplify further the heat kernel estimates we obtain. 
\end{remark}

\subsubsection{Cutting graphs}

Instead of gluing graphs together, we can also consider cutting a graph $(\Gamma, \pi, \mu)$ apart into pages and a spine. 

Given $(\Gamma, \pi, \mu),$ identify a set of vertices (and their induced subgraph) as $\Gamma_0.$ Remove the \textbf{vertices} in $\Gamma_0$ from $\Gamma.$ This splits the graph into connected components with trailing edges left from removing the vertices in $\Gamma_0.$ Assume $\Gamma_0$ is such that this results in a finite number of connected components, $\Gamma_1, \dots, \Gamma_l,$ all of which are infinite graphs. We ``cap'' the trailing edges in a given $\Gamma_i$ with vertices and call the set of these cap vertices $G_i.$ Then $G_i$ has a natural identification with a subset of vertices $G_i'$ of $\Gamma_0.$ 

On each of the pages $\Gamma_i$ created, we allow for a random walk structure with slightly more flexibility than simply taking the inherited Neumann Markov kernel: We require that there exist constants $C_B, c_B$ such that for all $x \in \Gamma_i,\ 0 \leq i \leq l,$
\begin{align}\label{compatible_cutting} c_B \pi(x) \leq \pi_i(x) \leq C_B \pi(x),\end{align}
and we require a similar inequality for edge weights. This flexibility ensures it is possible to cut a graph apart and glue it back together in such a way as to obtain \emph{exactly} the original graph; see Figure \ref{half_plane_cut}, which first appeared in \cite{FK_glue}, for an idea of why this is necessary. Again, this cutting procedure produces graphs for our pages which satisfy our basic hypotheses (see Lemma 3.3 of \cite{FK_glue}). 

We require the following of the above cutting construction:
\begin{enumerate}[a.]
\item As already mentioned, removing $\Gamma_0$ should create a finite number of connected components, all of which are infinite.
\item In $\Gamma,$ the graph $\Gamma_0$ is $\alpha$-connected (but see Remark \ref{alpha_conn}).
\item The set $\Gamma_0$ should have the property that any vertex $x \in \Gamma_0$ either has (1) all neighbors also in $\Gamma_0$ or (2) has neighbors in $\Gamma_i$ and $\Gamma_j$ with $i \not =j$ and $i, j \in \{0, 1, \dots, l\}.$ 
\end{enumerate}

Property c. above ensures that $\partial_I \Gamma_i = G_i$ and prevents selecting vertices for $\Gamma_0$ that are surrounded by other vertices from only one page and is the analog of property C. for the gluing operation. 

\begin{figure}[t]
\centering
\begin{tikzpicture}[scale=.8]
   \foreach \i in {0,...,5}{
           \filldraw[cb_blue] (\i,-1) circle (4pt);
        \filldraw[cb_blue] (\i,0) circle (4pt);
        \filldraw[cb_blue] (\i, 1) circle(4pt);
        \node at (\i-0.2, 1.3) {\small $8$};
          \node at (\i-0.2, 0.3) {\small $8$};
            \node at (\i-0.2, -0.7) {\small $8$};}
    \foreach \i in {0,...,5}{
       \draw[cb_blue,thick] (\i,-2)--(\i,-1)--(\i,0)--(\i,1)--(\i,2);}
      \draw[cb_blue,thick](-1,1)--(6,1);
    \draw[cb_blue,thick] (-1,0)--(6,0);
        \draw[cb_blue,thick] (-1,-1)--(6,-1);
    
    \node at (6, 0.5) {\textcolor{cb_blue}{$\Gamma$}};
    
    \foreach \i in {10,...,15}{
    \filldraw[cb_teal] (\i, 1) circle (4pt);
    \draw[cb_teal,thick] (\i,0.2)--(\i,2);}
    \draw[cb_teal, thick] (9,1)--(16,1);
    \node at (16, 1.5) {\textcolor{cb_teal}{$\Gamma_1$}};

        \foreach \i in {10,...,15}{
    \filldraw[cb_orange] (\i, -1) circle (4pt);
    \draw[cb_orange,thick] (\i,-0.2)--(\i,-2);}
    \draw[cb_orange, thick] (9,-1)--(16,-1);
        \node at (16, -1.5) {\textcolor{cb_orange}{$\Gamma_2$}};

%
    \foreach \i in {0,...,5}{
       \filldraw[cb_teal] (\i,-4) circle (4pt);
       \filldraw[cb_teal] (\i,-5) circle (4pt);
          \draw[cb_teal,thick] (\i,-5)--(\i,-4)--(\i,-3);
          \node at (\i-0.2, -3.7) {\small $8$};
             \node at (\i-0.2, -4.7) {\small $8$};}
       \draw[cb_teal, thick] (-1,-5)--(6,-5);
       \draw[cb_teal, thick] (-1,-4)--(6,-4);
           \node at (6, -3.5) {\textcolor{cb_teal}{$\Gamma_1$}};

    \foreach \i in {0,...,5}{
       \filldraw[cb_orange] (\i,-6) circle (4pt);
       \filldraw[cb_orange] (\i,-7) circle (4pt);
          \draw[cb_orange,thick] (\i,-6)--(\i, -8);
                    \node at (\i-0.2, -6.3) {\small $8$};
             \node at (\i-0.2, -7.3) {\small $8$};}
       \draw[cb_orange, thick] (-1,-6)--(6,-6);
       \draw[cb_orange, thick] (-1,-7)--(6,-7);
                  \node at (6, -7.5) {\textcolor{cb_orange}{$\Gamma_2$}};


                      \foreach \i in {10,...,15}{
       \filldraw[cb_teal] (\i,-4) circle (4pt);
       \filldraw[cb_teal] (\i,-5) circle (4pt);
          \draw[cb_teal,thick] (\i,-5)--(\i,-4)--(\i,-3);
          \node at (\i-0.2, -3.7) {\small $8$};
             \node at (\i-0.2, -4.7) {\small $4$};
             \node at (\i-0.5, -5.2) {\tiny $1/2$};}
       \draw[cb_teal, thick] (9,-5)--(16,-5);
   \draw[cb_teal, thick] (9,-4)--(16,-4);
        \node at (16, -3.5) {\textcolor{cb_teal}{$\Gamma_1$}};

    \foreach \i in {10,...,15}{
       \filldraw[cb_orange] (\i,-6) circle (4pt);
       \filldraw[cb_orange] (\i,-7) circle (4pt);
          \draw[cb_orange,thick] (\i,-6)--(\i, -8);
                    \node at (\i-0.2, -6.3) {\small $4$};
             \node at (\i-0.2, -7.3) {\small $8$};
              \node at (\i-0.5, -5.7) {\tiny $1/2$};}
      \draw[orange, thick] (9,-6)--(16,-6);
             \draw[cb_orange, thick] (9,-7)--(16,-7);
                \node at (16, -7.5) {\textcolor{cb_orange}{$\Gamma_2$}};

\end{tikzpicture}
 \caption{Begin with the lazy simple random walk on $\Z^2$ in the top left. Vertex weights are shown, and all edges have weight one. We cut it apart by removing the set of vertices where $y=0$, leaving us the two graphs with trailing edges on the top right. (The ``removed''  $\Gamma_0$ vertices are treated as a disconnected set all with weight zero, and are not shown.) By capping the trailing edges with vertices, adding back edges between these vertices, and taking the Neumann random walk, we obtain the bottom left pair of graphs with given vertex weights and all edge weights one. If these bottom left graphs are glued back together, one does not get precisely $\Z^2$ with lazy SRW. However, consider instead cutting apart the graphs where weights are assigned as in the bottom right, with all unlabeled edges having weight one. These bottom right graphs can be glued back together into $\Z^2$ with the lazy simple random walk.}\label{half_plane_cut}
 \end{figure}
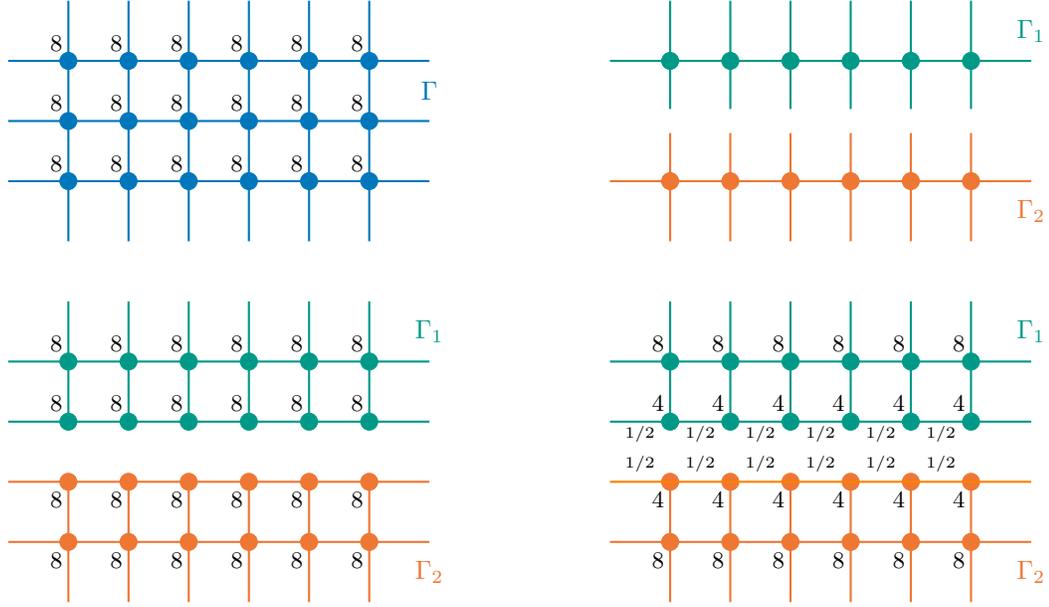

Example 3.5 of \cite{FK_glue} gives further explanation involving Figure \ref{half_plane_cut}, and Section 3.4 of \cite{FK_glue} provides many more examples. 

\subsubsection{Hypotheses on the spine}

In order to say anything about the heat kernel on such glued graphs as we have described, we need to know the spine behaves nicely in some sense. Two possibilities for the behavior of the spine are given below. 

\begin{definition}[Fixed width spine]\label{fixed_width}
Assume $(\Gamma,\pi,\mu)$ has pages $\Gamma_1, \dots, \Gamma_l$ and spine $\Gamma_0.$ The spine $\Gamma_0$ has a \emph{fixed width} of $\delta>0$ if, for all $v \in \Gamma_0,$ there exists $i \in \{1, \dots, l\}$ such that $d(v, \Gamma_i) \leq \delta.$ 
\end{definition}

\begin{definition}[Book-like graph]
Assume $(\Gamma,\pi,\mu)$ has pages $\Gamma_1, \dots, \Gamma_l$ and spine $\Gamma_0.$ For $\delta >0,$ we say $\Gamma$ is a \emph{$\delta$-book-like} graph (or simply a book-like graph) if, for all $v \in \Gamma_0$ and for all $1 \leq i \leq l,$ we have $d(v, \Gamma_i) \leq \delta.$
\end{definition}

If $\Gamma$ has a spine of fixed width $\delta,$ then each vertex in the spine $\Gamma_0$ is at distance at most $\delta$ from \emph{some} page; the nearby page may be different for different points in the spine. By contrast, if $\Gamma$ is a $\delta$-book-like graph, then each vertex in the spine $\Gamma_0$ is at distance at most $\delta$ from \emph{all} pages. If $\Gamma$ is a $\delta$-book-like graph, then its spine also has fixed width $\delta$. The term ``book-like'' is because the spine can see all pages, as is the case for the spine of a book. 

The main results in this paper hold for book-like graphs. We therefore often say ``$\Gamma$ is a book-like graph'' to mean $\Gamma$ is a graph with pages $\Gamma_1, \dots, \Gamma_l$ and spine $\Gamma_0$ satisfying the construction described in this section that also satisfies the $\delta$-book-like property. 

It is often useful to consider what we call ``augmented'' pages. 

\begin{definition}[Augmented page(s)]\label{aug_pages}
Let $(\Gamma, \pi, \mu)$ be a graph of fixed width $\delta$ with pages $\Gamma_1, \dots, \Gamma_l$ and spine $\Gamma_0$. For each $i \in \{1,\dots, l\},$ the \emph{augmented page} associated with $\Gamma_i$, denoted $\widehat{\Gamma}_i$, is 
\[ \widehat{\Gamma}_i := [\Gamma_i]_\delta \cap (\Gamma_i \cup \Gamma_0),\]
where $[\Gamma_i]_\delta := \{ y \in \Gamma: d(y, \Gamma_i) \leq \delta\}$ is the $\delta$-neighborhood of $\Gamma_i.$ 

A random walk structure on the augmented page $\widehat{\Gamma}_i$ is given by setting $\widehat{\pi}_i,\ \widehat{\mu}^i$ to take the values of $\pi, \mu$ from $\Gamma.$ In the event that the weights on $\Gamma_i,\ \Gamma_0$ are not precisely the same as those of $\Gamma$ (due to the compatible weights condition), take instead the weights from $\Gamma_i$ and $\Gamma_0$ (seen as graphs separate from $\Gamma$). The precise choice of weight does not alter the results obtained here except up to a change of constants. 
\end{definition} 

If $\Gamma$ is $\delta$-book-like, then $\Gamma_0 \subseteq \widehat{\Gamma}_i$ for all $i \in \{1, \dots, l\}.$ In other words, the augmented page $\widehat{\Gamma}_i$ consists of page $i$ along with the entire spine.

\subsection{Key hypotheses}\label{hyp}

We may now state the hypotheses for our main results. 

\begin{enumerate}
\item[(B1)] $(\Gamma, \pi, \mu)$ is a $\delta$-book-like graph with pages $\Gamma_1, \dots, \Gamma_l$ and spine $\Gamma_0$.

\item[(B2)] Each page $\Gamma_i, \ 1 \leq i \leq l$ is a Harnack graph (Definition \ref{Harnack_graph}). 

\item[(B3)] Each page $\Gamma_i$ is inner uniform (Definition \ref{innerunif}) and is uniformly $S$-transient (Definition \ref{S_trans_def}) when considered as a subgraph of the augmented page $\widehat{\Gamma}_i$ (Definition \ref{aug_pages}). 

\item[(B4)] Each page $\Gamma_i, \ 1 \leq i \leq l$ is uniform (Definition \ref{unif}) in $\Gamma.$ 
\end{enumerate}

It is worth stating here that each page $\Gamma_i$ is quasi-isometric to its augmented version $\widehat{\Gamma}_i$ (see Lemma 4.1 of \cite{FK_glue}). Since each page $\Gamma_i$ is Harnack, this means each augmented page $\widehat{\Gamma}_i$ is also Harnack by a result of Coulhon and Saloff-Coste \cite[Theorem 8.3]{TC_LSC_IsoInfini}.

An important consequence of the (inner) uniformity assumptions is that we may replace the distance $d_i := d_{\Gamma_i}$ taken in a page $\Gamma_i$ with the distance $d := d_\Gamma$ taken in the entire book-like graph.

\section{Gluing heat kernels}\label{glue_setup} 

The objective of this section is to build up the main tools needed to obtain heat kernel estimates in the next section. We begin with some abstract estimates for gluing heat kernels that have minimal hypotheses. Then, given (B1)-(B4), we state what we know about the terms appearing in these abstract estimates. 

\subsection{Abstract gluing estimates}

The following theorem is a discrete version of Theorem 3.5 of \cite{lsc_ag_ends} and does not require any of the hypotheses (B1)-(B4). 

\begin{theorem}\label{gluing_thm} Let $U_1, U_2$ be two subgraphs of $(\Gamma, \pi, \mu)$ satisfying one of the following two conditions:
\begin{enumerate}
\item $U_1 \cap U_2 = \emptyset$ in such a way that $\partial U_1 \cap U_2$ and $\partial U_2 \cap U_1$ are also empty
\item $U_2 \subset U_1.$ 
\end{enumerate}
Then for all $x \in U_1, y \in U_2,$ and $n\geq d(x,y),$ 
\begin{align}
p(n,x,y) &\leq p_{U_1,D}(n,x,y)
\\&+\underbrace{2 \sum_{v \in \partial U_1} \sum_{w \in \partial U_2} \sup_{\lfloor \frac{n}{4} \rfloor \leq m \leq n} p(m,v,w) \psi_{\partial U_1}(n,x,v) \psi_{\partial U_2} (n,y,w)}_{\text{{\color{cb_blue}TERM A}}} \\
&+ \underbrace{\sum_{v \in \partial U_1} \sum_{w \in \partial U_2} \sup_{\lfloor \frac{n}{4} \rfloor \leq m \leq n} \psi'_{\partial U_2}(m,y,w) \psi_{\partial U_1}(n,x,v) \sum_{l=d(v,w)}^n p(l,v,w)}_{{\color{cb_teal}\text{TERM } B_x}} \\
&+ \underbrace{\sum_{v \in \partial U_1} \sum_{w \in \partial U_2} \sup_{\lfloor \frac{n}{4} \rfloor \leq m \leq n} \psi'_{\partial U_1}(m,x,v) \psi_{\partial U_2}(n,y,w) \sum_{l=d(v,w)}^n p(l,v,w)}_{{\color{cb_magenta}\text{TERM } B_y}} .
\end{align}
and 
\begin{align}
2p(n,x,y) & \geq p_{U_1, D}(n,x,y)
\\&+\underbrace{2\sum_{v \in \partial U_1} \sum_{w \in \partial U_2} \inf_{\lfloor \frac{n}{4} \rfloor \leq m \leq n} p(m,v,w) \psi_{\partial U_1} \Big(\Big\lfloor \frac{n}{4}\Big\rfloor, x, v\Big) \psi_{\partial U_2} \Big(\Big\lfloor\frac{n}{4}\Big\rfloor, y, w\Big)}_{{\color{cb_blue}\text{TERM }A}} \\
 &+ \underbrace{\sum_{v\in \partial U_1} \sum_{w \in \partial U_2} \inf_{\lfloor \frac{n}{4}\rfloor \leq m \leq n} \psi'_{\partial U_2}(m,y,w) \psi_{\partial U_1} \Big( \Big\lfloor \frac{n}{4} \Big\rfloor, x, v\Big) \sum_{l=d(v,w)}^{\lfloor \frac{n}{4}\rfloor -1} p(l,v,w)}_{{\color{cb_teal}\text{TERM } B_x}} \\
 & + \underbrace{\sum_{v \in \partial U_1} \sum_{w \in \partial U_2} \inf_{\lfloor \frac{n}{4} \rfloor \leq m \leq n} \psi'_{\partial U_1}(m,x,v) \psi_{\partial U_2} \Big( \Big\lfloor \frac{n}{4} \Big\rfloor, y, w \Big) \sum_{l=d(v,w)}^{\lfloor \frac{n}{4} \rfloor -1} p(l,v,w)}_{{\color{cb_magenta}\text{TERM } B_y}}.
\end{align}

\end{theorem}

The hitting probabilities $\psi_{\partial U_1}(n,x,v)$ and $\psi'_{\partial U_1}(m,x,v)$ are as in Definition \ref{hitprob_def}. The proof of Theorem \ref{gluing_thm} is essentially the same as in the continuous and compact case of \cite{lsc_ag_ends} and is based on a series of gluing lemmas. The full details of this proof are found in Appendix A of the first author's PhD thesis \cite{Emily_thesis}. In the setting of gluing two copies of $\R^n$ via a paraboloid of revolution in \cite{ag_ishiwata}, Grigor'yan and Ishiwata use Lemma 4.1 as their main gluing lemma. Their Lemma 4.1 can be seen as a continuous version of Lemmas 3.1 and 3.3 of \cite{lsc_ag_ends} that can handle gluing over non-compact (unbounded sets). Since their setting only ever has two ends, it suffices to use a generalization of the gluing lemmas of \cite{lsc_ag_ends}; our present setting with an arbitrary finite number of pages requires the above generalization of the full gluing theorem (Theorem 3.5) of \cite{lsc_ag_ends}.

Although the sums appearing in Theorem \ref{gluing_thm} appear to be over the full (possibly infinite) sets $\partial U_1,\ \partial U_2,$ in reality each of the three ``colorful'' terms is only non-zero when all of the following are true: $d_\Gamma(x,v) \leq n, \ d_\Gamma(y,w) \leq n,$ and $d_\Gamma(v,w) \leq n$. Consequently all sums are finite, and the distance between any points appearing is at most of order $n$. 

The additional hypotheses (B1)-(B4) and gluing structure let us estimate the objects appearing in the above estimate. In general, we apply Theorem \ref{gluing_thm} with either $U_1, \ U_2$ being distinct pages or with $U_1$ being a neighborhood of a page and $U_2$ being the page itself. These two settings correspond to the two conditions on $U_1,\ U_2$ in the theorem. 

The careful reader might be concerned that such choices of $U_1, \ U_2$ in the paragraph above need not satisfy the conditions on these sets given in the theorem. For instance, if $U_1 = \Gamma_1$ and $U_2 = \Gamma_2$, it is possible that $\partial \Gamma_1 \subset \Gamma_2$. Indeed, this is the case in our main example of gluing a copy of $\Z^4$ one of $\Z^5$ and one of $\Z^6$ by identifying their $x_1$-axes. However, in Section \ref{abstract_est} below, we will see this difficulty can be resolved by taking $U_1$ to be say $\Gamma_1$ minus a neighborhood of the spine. In general, sums of the sort appearing in the above gluing theorem are fairly flexible. Therefore, without loss of generality, assume the sums are over points that only belong to the spine. Then there are two kinds of objects in the gluing theorem: (1) heat kernels between two points in the gluing spine, e.g. $p(m,v,w)$ and sums thereof, and (2) hitting probabilities of the gluing set/exit probabilities of a page.  We obtained estimates for spine-to-spine heat kernels (1) using Faber-Krahn inequalities in \cite{FK_glue}, and estimates for the hitting probabilities (2) in \cite{ed_lsc_trans}.

\subsection{Estimates of quantities in Theorem \ref{gluing_thm}}\label{term_estimates}

Theorem \ref{gluing_thm} provides a way to obtain heat kernel estimates with upper and lower bounds that are matching (modulo some supremums/infimums). In this subsection, we consolidate results of our earlier papers to give two-sided estimates for all of the objects appearing in Theorem \ref{gluing_thm} in terms of more concrete quantities such as distances and volumes. 

In Theorem \ref{gluing_thm}, the spine-to-spine heat kernel $p(m,z,w)$ where $z,w \in \Gamma_0$ appears, as well as finite sums of $p(m,z,w)$ in the time variable. Estimates for the spine-to-spine heat kernel were obtained in \cite{FK_glue} and are summarized below. In general, as discussed in Section 7 of \cite{FK_glue}, obtaining good $\Gamma_0$ to $\Gamma_0$ heat kernel estimates, particularly in the upper bound, is a major obstacle to proving more general results about graphs that are not book-like.

For any $v \in \Gamma_0,$ define
\[ V_{\min}(v,r) = \min_{1 \leq i \leq l} V_i(v_i, r),\]
where each $v_i$ is a closest point to $v$ in $\Gamma_i$ and $V_i := V_{\Gamma_i}$. (In \cite{FK_glue}, we gave a more complicated definition of $V_{\min}$. However, in the present setting of book-like graphs, the two definitions of $V_{\min}$ are comparable.) As shown in Section 6.1 of \cite{FK_glue}, under the assumption the graph is book-like, $V_{\min}(v,r)$ is a doubling function that increases in $r$ for fixed $v$.

\begin{lemma}[Theorems 6.3 and 6.5 of \cite{FK_glue}, spine-to-spine heat kernel estimate]\label{spine_to_spine}
Let $(\Gamma, \mu, \pi)$ be a book-like graph satisfying (B1)-(B4) of Section \ref{hyp}. Then for all $v, w \in \Gamma_0$ and $m \gg d_\Gamma(v,w) + \delta,$
\begin{equation}\label{spine_to_spine_est}
p(m,v,w) \approx \frac{C}{V_{\min}(v, \sqrt{m})} \exp\Big(-\frac{d_\Gamma^2(v, w)}{cm}\Big),
\end{equation}
where the notation $\approx$ is abused as in Definition \ref{approx_def2}.
\end{lemma}

\begin{lemma}[Sums of spine-to-spine heat kernel estimates]\label{spine_to_spine_sum}
Let $(\Gamma, \mu, \pi)$ be a book-like graph satisfying hypotheses (B1)-(B4). 

Then for all $v \not = w \in \Gamma_0$ and $n \gg d_\Gamma(v,w) + \delta$,
\begin{align}
\begin{split}\label{timesum_big}
\sum_{l= d_\Gamma(v,w)}^n p(l,v,w) &\approx \sum_{l = d_\Gamma(v,w)}^n \frac{C}{V_{\min}(v,\sqrt{l})} \exp\Big(-\frac{d_\Gamma^2(v,w)}{cl}\Big) \\&\approx C\frac{d_\Gamma^2(v,w)}{V_{\min}(v, d_\Gamma(v,w))}\exp\Big(-\frac{d_\Gamma^2(v,w)}{cn}\Big) + \sum_{l = d_\Gamma^2(v,w)}^n \frac{C}{V_{\min}(v, \sqrt{l})}.
\end{split}
\end{align}

If $v=w$ and $n \gg d_\Gamma(v,w) + \delta$, then 
\[ \sum_{l=0}^n p(l,v,v) \approx C + \sum_{l=C^* \delta}^n \frac{C}{V_{\min}(v,\sqrt{l})}\]
for some constant $C^*$. (Note $C$ depends on $\delta$.) 
\end{lemma}

Lemma \ref{spine_to_spine_sum} can be proved using the same kind of arguments as in \cite{ed_lsc_trans} (see the proof of Corollary 3.24). The proof considers different cases on the size of $n$ compared to $d_\Gamma(v,w)$ and utilizes the standard technique of taking advantage of the doubling function $V_{\min}$ using a dyadic decomposition. If $v=w,$ then $p(0,v,v)=1$ and indeed $p(l,v,v)$ is roughly $1$ for small $l$; for large $l,$ Lemma \ref{spine_to_spine} can be applied. 

Only the first term of (\ref{timesum_big}) is present if the graph also satisfies the following volume growth condition.

\begin{definition}[Volume growth condition]\label{vol_growth_cond} On a book-like graph $(\Gamma, \mu, \pi)$, we say the volume grows faster than $2$ if for all $1 \leq i \leq l,$ there exists $\varepsilon_i >0,$ such that 
\begin{equation}\label{minvol_faster2}
\frac{V_i(x,R)}{V_i(x,r)} \geq c_i \, \bigg(\frac{R}{r}\bigg)^{2 + \varepsilon_i} \quad \forall x \in \Gamma_i, \ R \geq r.
\end{equation}
In particular, this implies that there exists $\varepsilon >0$ such that
\[ \frac{V_{\min}(v,R)}{V_{\min}(v,r)} \geq c \Big(\frac{R}{r}\Big)^{2+\varepsilon} \quad \forall \nu \in \Gamma_0,\ R \geq r\]
as well. 
\end{definition}

Let $\rho_{\Gamma}(x,y) := \max \{1, d_\Gamma(x,y)\}.$ If (\ref{minvol_faster2}) is satisfied, then the estimate (\ref{timesum_big}) reduces to 
\begin{equation}\label{reduced_timesum}
\sum_{l= d_\Gamma(v,w)}^n p(l,v,w) \approx C \frac{\rho_\Gamma^2(v,w)}{V_{\min}(v, d_\Gamma(v,w))}\exp\Big(-\frac{d_\Gamma^2(v,w)}{cn}\Big).
\end{equation} 
This simplification is obvious in the case that $n \leq d_\Gamma^2(v,w)$ and for the lower bound. In the remaining case, use the doubling property of $V_{\min}$ to compute the ``tail'' sum in (\ref{timesum_big}). The condition (\ref{minvol_faster2}) guarantees the sum is finite. The use of $\rho_\Gamma$ in the numerator handles the case where $v=w$ so that the estimate is a constant, as opposed to zero. In all other situations $\rho_\Gamma$ and $d_\Gamma$ are interchangeable.

The other objects appearing in Theorem \ref{gluing_thm} are hitting probabilities. Applying results of \cite{ed_lsc_trans} to the present setting gives the following lemma. 

\begin{lemma}[Corollaries 3.23 and 3.24 of \cite{ed_lsc_trans}]\label{psi_summary} 
Let $(\widehat{\Gamma}, \pi,\mu)$ be a connected graph that is uniformly lazy with controlled weights. Assume that $(\widehat{\Gamma}, \pi,\mu)$ is Harnack and $\Gamma$ is an inner uniform subgraph of $\widehat{\Gamma}$ that is uniformly $S$-transient. Then, $\forall x \in \Gamma \setminus \partial_I \Gamma, \ v \in \partial \Gamma, \ n \geq d_\Gamma(x,v),$ with $\approx$ as in Definition \ref{approx_def2},
\begin{align}\label{hitprob_deriv_nice}
&\psi_{\partial \Gamma}' (n,x,v) \approx \frac{C \pi(v)}{V_\Gamma(x, \sqrt{n})} \exp\Big(-\frac{d_\Gamma^2(x,v)}{c n}\Big) \\[1ex]\label{hitprob_big}
&\psi_{\partial \Gamma}(n,x,v) \approx \frac{C \pi(v)\rho_\Gamma^2(x,v)}{V_\Gamma(x, d_\Gamma(x,v))}\exp\Big(-\frac{d_\Gamma^2(x,v)}{c n}\Big) + \sum_{m=d_\Gamma^2(x,v)}^n \frac{C \pi(v)}{V_\Gamma(x,\sqrt{m})}.
\end{align}

If, in addition, $V_\Gamma$ satisfies condition (\ref{minvol_faster2}), then
\begin{align}
\psi_{\partial \Gamma}(n,x,v) \approx \frac{C \pi(v)\rho_\Gamma^2(x,v)}{V_\Gamma(x, d_\Gamma(x,v))}\exp\Big(-\frac{d_\Gamma^2(x,v)}{c n}\Big).
\end{align}
\end{lemma}

We will in general apply Lemma \ref{psi_summary} in the case where $\Gamma$ is a page and $\widehat{\Gamma}$ is the augmented page or a neighborhood of the page. Since these two cases look somewhat different, this lemma is left in fairly general subgraph terms.

\section{General heat kernel estimates}\label{abstract_est}

We now carry out the strategy mentioned in the previous section of applying Theorem \ref{gluing_thm} and the estimates of the objects appearing therein. This gives us our main general heat kernel estimate for book-like graphs.

Before the statement of our main theorem, we need to set some additional notation. For $1 \leq i \leq l,$ set $V_i := V_{\Gamma_i}$ and $ d_i := d_{\Gamma_i}$. If $x, y$ do not both belong to $\Gamma_i,$ then set $d_i(x,y) := + \infty$. Also set $\rho_i (x,y) := \max \{ 1, d_i(x,y)\}$. 

Recall the augmented pages are defined as $\widehat{\Gamma}_i := \Gamma_i \cup \Gamma_0$ for all $1 \leq i \leq l$. That is, they are the union of the page and the spine. For $x \in \widehat{\Gamma}_i, \ y \in \widehat{\Gamma}_j,$ where $i$ and $j$ may be the same or distinct, define
\[ P_{x,y,n} = \{ (v,w) \in \partial \Gamma_{i} \times \partial \Gamma_{j} : d(x,v) \leq n, d(y, w) \leq n, d(v,w) \leq n\}.\] 
If say $x \in \Gamma_0,$ then it belongs to $\widehat{\Gamma}_i$ for all $1 \leq i \leq l,$ and therefore the definition of $P_{x,y,n}$ is dependent upon the statement that $x \in \widehat{\Gamma}_i$ to determine $i$. If instead $x$ belongs to only one page, then the choice of $i$ must be the index of the page containing $x$.

Finally, given four vertices $v,w,x,y,$ set
\[ d_{*}^2 := d_\Gamma^2(x,v) + d_\Gamma^2(v,w) + d_\Gamma^2(w,y).\]
Since $(a+b)^2 \approx a^2 +b^2$ when $a, b \geq 0,$ it is also true that 
\[ d_*^2 \approx \big(d_\Gamma(x,v)+d_\Gamma(v,w)+d_\Gamma(w,y)\big)^2.\]

\begin{theorem}\label{main_general_thm}
Assume $(\Gamma,\pi,\mu)$ is a book-like graph with pages $\Gamma_1, \dots, \Gamma_l$ and spine $\Gamma_0$ satisfying hypotheses (B1)-(B4).  Also assume that the volume growth condition of Definition \ref{vol_growth_cond} is satisfied.

Then we have the following heat kernel estimates for $x,y \in \Gamma$ and $n \gg d_\Gamma(x,y) + d(x, \Gamma_0) + d(y, \Gamma_0) + \delta$: 

\begin{enumerate}
    \item If $x \in \widehat{\Gamma}_i, \ y \in \widehat{\Gamma}_j$ ($i, \ j$ may be distinct or not) we have upper and lower bounds on $p(n,x,y)$ of the form
\begin{align}
\begin{split}\label{main_est}
&\frac{C}{V_{i}(x, \sqrt{n})} \exp\Big(-\frac{d_{i}^2(x,y)}{cn}\Big) \\
    &+\sum_{(v,w) \in P_{x,y,n}} C\Bigg[\frac{1}{V_{\min}(v, \sqrt{n})} \frac{\pi(v) \rho_i^2(x,v)}{V_i(x, d_i(x,v))}\frac{\pi(w) \rho_j^2(y,w)}{V_j(y, d_j(y,w))}
\\
&+\frac{\pi(w)}{V_j(y, \sqrt{n})} \frac{\pi(v) \rho_i^2(x,v)}{V_i(x, d_i(x,v))} \frac{\rho_\Gamma^2(v,w)}{V_{\min}(v, d_\Gamma(v,w))} \\
&+\frac{\pi(v)}{V_i(x, \sqrt{n})} \frac{\pi(w) \rho_j^2(y,w)}{V_j(y, d_j(y,w))} \frac{\rho_\Gamma^2(v,w)}{V_{\min}(v, d_\Gamma(v,w))} \Bigg] \exp\Big(-\frac{d_*^2}{cn}\Big)
\end{split}
\end{align}

\item If $x, \ y$ are both near the spine, then (\ref{main_est}) reduces to 
\begin{align}\label{central_est} 
p(n,x,y) \approx \frac{C}{V_{\min}(x, \sqrt{n})} \exp\Big(-\frac{d_\Gamma^2(x,y)}{cn}\Big). 
\end{align}

\end{enumerate}

\end{theorem} 

\begin{remark}

\begin{itemize}
\item We made the volume growth assumption (\ref{vol_growth_cond}) out of convenience as the formulas in the theorem are already sufficiently complicated. Without this assumption, our methods and proofs still work by simply using the full estimates given in the previous section. In Example \ref{Z_D_plane} below, one page does not satisfy the volume growth condition. 
\item Technically, in the lower bound, the set $P_{x,y,n}$ should be replaced with $P(x,y,\lfloor\frac{n}{4}\rfloor)$, but due to the doubling property of all volumes present and the ability to change the constants $C, c$, this makes no difference. Similarly, whether $n$ or $\lfloor \frac{n}{4}\rfloor$ is used in any of the terms inside the sum is also irrelevant. 
\item In the places in the formula where $d$ remains instead of $\rho,$ one may always replace with $\rho$ if desired (again due to volume doubling and changes up to constants). 
\item If $x$ (or $y$) is in the spine, then (\ref{main_est}) holds with any choice of $i$ in $\{1, \dots, l\}$. The courageous reader can verify that the estimates in (\ref{main_est}) are the same regardless of choice of $i$, and also that if both $x$ and $y$ belong to the spine, then (\ref{central_est}) is also the same. (This is easier to verify in the specific examples given in later sections.) 
\end{itemize}
\end{remark}

\begin{proof}

First, notice estimate (\ref{central_est}) is exactly the spine-to-spine heat kernel estimate seen above in (\ref{spine_to_spine_est}).  Moreover, by the local parabolic Harnack inequality, $p(n, x,y) \approx Cp(cn, x^*,y^*),$ where $x^*$ and $y^*$ are are most a fixed distance from $x$ and $y$, respectively. Therefore it suffices to prove the theorem for $x, \ y,$ satisfying $d(x, \Gamma_0), d(y, \Gamma_0) > 4.$ There are two cases based on whether $x, \ y$ are in different pages or the same page. 

The following lemma should reassure the reader that the estimate (\ref{main_est}) is not sensitive to small changes.

\begin{lemma}\label{notsens_sum}
    In (\ref{main_est}), the sum over $(v,w) \in P_{x,y,n}$ can be replaced with a sum over any of the following sets:
    \begin{itemize}
        \item $\{(v,w) \in [\Gamma_0]_{K} \times [\Gamma_0]_K : d(x,v) \leq n, d(y,w) \leq n, d(v,w) \leq n\} $ where $K$ is any fixed constant (say $K=4$ or $K=10$)
        \item sum over the centers of balls of any fixed $\varepsilon$-net of any of the sets of the previous item
        \item either of the previous two items, but with $n$ replaced by $\lfloor kn \rfloor$ for some constant $k >0$ (in particular $\lfloor \frac{n}{4} \rfloor$)
        \item any of the previous sets but with $x, y$ replaced with $x^*, y^*$ where $d_\Gamma(x, x^*), \allowbreak  d_\Gamma(y, y^*) \leq K$.
    \end{itemize}
\end{lemma}

\begin{proof}[Proof of Lemma \ref{notsens_sum}]
    We prove the equivalence between summing over the entire $K$-neighborhood of the spine and summing over the centers of a $\varepsilon$-net of this neighborhood, as from there the other statements are obvious.

    Fix $\varepsilon >0$ and take a covering of  $\{(v,w) \in [\Gamma_0]_{K} \times [\Gamma_0]_K : d(x,v) \leq n, d(y,w) \leq n, d(v,w) \leq n\}$ by balls of radius $\varepsilon$ such that the balls of radius $\frac{\varepsilon}{2}$ are disjoint. Such a covering always exists since the graph is locally uniformly finite and each page is volume doubling. 

    Certainly summing over only the centers of the balls in the $\varepsilon$-net is smaller than summing over the entire neighborhood of the spine. On the other hand, for every pair $(v,w) \in \{(v,w) \in [\Gamma_0]_{K} \times [\Gamma_0]_K : d(x,v) \leq n, d(y,w) \leq n, d(v,w) \leq n\}$, there exists a pair of balls $B(v^*, \varepsilon), B(w^*, \varepsilon)$ in the covering such that $v \in B(v^*, \varepsilon), w \in B(w^*, \varepsilon)$. Due to volume doubling and controlled weights, all terms inside of the sum with $v, \ w$ can be replaced by the same term with $v^*, w^*$. Moreover, each $(v^*, w^*)$ only represents at most a fixed finite number of $(v,w)$ pairs. 
\end{proof}

We now return to the main proof.

\noindent\emph{Case 1: $x, \ y$ are in distinct pages away from the gluing set} 

Assume $x \in \Gamma_i,\ y \in \Gamma_j$ where $i \not = j$ and $d(x, \Gamma_0),\ d(y, \Gamma_0) > 4.$ If the spine is thick enough that $\partial \Gamma_i \cap \Gamma_j = \emptyset$ and $\partial \Gamma_j \cap \Gamma_i = \emptyset,$ then we may apply Theorem \ref{gluing_thm} with $U_1 = \Gamma_i,\ U_2= \Gamma_j$ since condition (1) of the theorem is satisfied. If this condition does not hold, instead apply Theorem \ref{gluing_thm} with $U_1 = \Gamma_i \setminus [\Gamma_0]_4$ and $U_2 = \Gamma_j \setminus [\Gamma_0]_4$. The difference in applying the theorem in these two situations is what set the sums are over; due to Lemma \ref{notsens_sum} above, it suffices to consider the first situation. 

Since $x$ and $y$ are in different pages, $p_{\Gamma_i, D}(n,x,y) = 0$ for all $n$. We are left with the three ``colorful'' terms to estimate. Since $V_{\min}$ is doubling, the spine-to-spine heat kernel estimates of $p(m,v,w)$ given in (\ref{spine_to_spine_est}) do not change when taking supremums or infimums with $\lfloor \frac{n}{4}\rfloor \leq m \leq n$. The hitting probabilities of interest are of the type $\psi_{\partial \Gamma_i}(n,x,v)$ and $\psi'_{\partial \Gamma_i}(m,x,v)$. These can be estimated using Lemma \ref{psi_summary} since by hypotheses (B3), each page $\Gamma_i$ is an inner uniform and uniformly $S$-transient subgraph of the augmented page $\widehat{\Gamma}_i.$ Moreover, hypothesis (B2) implies $\widehat{\Gamma}_i$ is Harnack. Additionally, any distances appearing in Lemma \ref{psi_summary} that ought to be taken in either $\Gamma_i$ or $\widehat{\Gamma}_i$ can be taken instead in the full graph $\Gamma$ due to the uniformity hypothesis (B4).

Substituting the estimates from the previous paragraph into Theorem \ref{gluing_thm} gives the following matching upper/lower bounds on the terms (the first line is written as in the upper bound, but a similar statement holds in the lower bound): 
\begin{align*}
\text{{\color{cb_blue}TERM A}} &\approx 2 \sum_{v \in \partial \Gamma_i} \sum_{w \in \partial \Gamma_j} \sup_{\lfloor \frac{n}{4} \rfloor \leq m \leq n} p(m,v,w) \psi_{\partial \Gamma_i}(n,x,v) \psi_{\partial \Gamma_j} (n,y,w) \\
&\approx \sum_{(v,w) \in P_{x,y,n}} \frac{C}{V_{\min}(v, \sqrt{n})} \frac{\pi(v)\, \rho_i^2(x,v)}{V_i(x, d_i(x,v))}\frac{\pi(w)\, \rho_j^2(y,w)}{V_j(y, d_j(y,w))} \exp\Big(-\frac{d_*^2}{cn}\Big)
\end{align*}
\begin{align*}
{\color{cb_teal}\text{TERM } B_x} &\approx \sum_{v \in \partial \Gamma_i} \sum_{w \in \partial \Gamma_j} \sup_{\lfloor \frac{n}{4} \rfloor \leq m \leq n} \psi'_{\partial U_2}(m,y,w) \psi_{\partial U_1}(n,x,v) \sum_{l=d(v,w)}^n p(l,v,w) \\
&\approx 
\sum_{(v,w) \in P_{x,y,n}} \frac{C\pi(w)}{V_j(y, \sqrt{n})} \frac{\pi(v) \rho_i^2(x,v)}{V_i(x, d_i(x,v))} \frac{\rho_\Gamma^2(v,w)}{V_{\min}(v, d_\Gamma(v,w))} \exp\Big(-\frac{d_*^2}{cn}\Big)
\end{align*} 
\begin{align*}
{\color{cb_magenta}\text{TERM } B_y} &\approx 
\sum_{v \in \partial \Gamma_i} \sum_{w \in \partial \Gamma_j} \sup_{\lfloor \frac{n}{4} \rfloor \leq m \leq n} \psi'_{\partial U_1}(m,x,v) \psi_{\partial U_2}(n,y,w) \sum_{l=d(v,w)}^n p(l,v,w) \\
&\approx 
\sum_{(v,w)\in P_{x,y,n}} \frac{C\pi(v)}{V_i(x, \sqrt{n})} \frac{\pi(w) \rho_j^2(y,w)}{V_j(y, d_j(y,w))} \frac{\rho_\Gamma^2(v,w)}{V_{\min}(v, d_\Gamma(v,w))} \exp\Big(-\frac{d_*^2}{cn}\Big).
\end{align*}

These are exactly the three terms appearing in the second part of (\ref{main_est}), and the first term of (\ref{main_est}) is zero since $x, y$ belong to different pages.

\noindent \emph{Case 2: $x, \ y$ are away from the same spine and in the same page}

Assume $x, \ y \in \Gamma_i$ and $d(x,\Gamma_0), d(y, \Gamma_0) > 4.$ We apply Theorem \ref{gluing_thm} with $U_2 = \Gamma_i$ and $U_1 = [\Gamma_i]_1$ (the $1$-neighborhood of $\Gamma_i$), sets which satisfy condition (2) of the theorem. It suffices to consider the situation that $\widehat{\Gamma}_i \supseteq [\Gamma_i]_1$ (otherwise take a bigger neighborhood of the page and use Lemma \ref{notsens_sum}). Then  since $\Gamma_i$ is inner uniform and uniformly $S$-transient inside a Harnack graph $\widehat{\Gamma}_i,$ it is also true that $\Gamma_i$ is inner uniform and uniformly $S$-transient inside the Harnack graph $[\Gamma_i]_1.$ Therefore, by Lemma \ref{D_approx_N},
\begin{align*}
p_{\Gamma_i, D}(n,x,y) \approx Cp_{\Gamma_i,N}(cn, x,y) \approx \frac{C}{V_{i}(x, \sqrt{n})} \exp\Big(-\frac{d_{i}^2(x,y)}{cn}\Big),
\end{align*}
since $\Gamma_i$ is itself Harnack with Neumann condition. 

Estimating the ``colorful'' terms is very similar to the first case. The hitting probabilities of $\partial \Gamma_i$ and of $\partial [\Gamma_i]_1$ are essentially the same. Again, the uniformity hypothesis (B4) lets distances in pages be replaced by distances in $\Gamma$. Therefore,
\begin{align*}
\text{{\color{cb_blue}TERM $A$}} &\approx 2 \sum_{v \in \partial [\Gamma_i]_1} \sum_{w \in \partial \Gamma_i} p(n,v,w) \psi_{\partial [\Gamma_i]_1}(n,x,v) \psi_{\partial \Gamma_i}(n, y, w) \\
&\approx \sum_{(v,w) \in P_{x,y,n}} \frac{C}{V_{\min}(v, \sqrt{n})} \frac{\pi(v) \rho_{i}^2(x,v)}{V_{i}(x, d_{i}(x,v))} \frac{\pi(w) \rho_{i}^2(y,w)}{V_{i}(y, d_{i}(y,w))} \exp\Big(-\frac{d_*^2}{cn}\Big).
\end{align*}
To understand the second approximation, observe quantities like $d_{\Gamma_i}(x,v)$ make sense when interpreted as extending $d_i$ to $[\Gamma_i]_2$. Moreover, due to uniformity, all such quantities can be replaced with with $d_\Gamma(x,v).$ Finally, the vertices summed over can be replaced with $P_{x,y,n}$ by Lemma \ref{notsens_sum}. 

The other two terms are estimated similarly:
\begin{align*}
{\color{cb_teal}\text{TERM } B_x} &\approx \sum_{v\in \partial [\Gamma_i]_1} \sum_{w \in \partial \Gamma_i} \psi'_{\partial \Gamma_i}(m,y,w) \psi_{\partial [\Gamma_i]_1} (n,x,v) \sum_{l=d(v,w)}^{n} p(l,v,w) \\
&\approx \sum_{(v,w) \in P_{x,y,n}} \frac{C \pi(w)}{V_i(y, \sqrt{n})} \frac{\pi(v) \rho_i^2(x,v)}{V_i(x, d_i(x,v))} \frac{\rho_\Gamma^2(v,w)}{V_{\min}(v, d_\Gamma(v,w))} \exp\Big(-\frac{d_*^2}{cn}\Big) 
\end{align*}
\begin{align*}
{\color{cb_magenta}\text{TERM } B_y} &\approx \sum_{v \in \partial [\Gamma_i]_1} \sum_{w \in \partial \Gamma_i} \psi'_{\partial [\Gamma_i]_1}(n,x,v) \psi_{\partial \Gamma_i} (n, y, w) \sum_{l=d(v,w)}^{n} p(l,v,w)\\
&\approx \sum_{(v,w) \in P_{x,y,n}} \frac{C\pi(v)}{V_i(x, \sqrt{n})} \frac{\pi(w) \rho_i^2(y,w)}{V_i(y, d_i(y,w))} \frac{\rho_\Gamma^2(v,w)}{V_{\min}(v, d_\Gamma(v,w))} \exp\Big(-\frac{d_*^2}{cn}\Big).
\end{align*}
Summing up all four terms found above gives precisely (\ref{main_est}) as in this case $i=j.$
\end{proof}  

\section{Example: The Spine \texorpdfstring{$\Gamma_0$}{} is finite}\label{finiteglue}

In this section, we consider the case of a book-like graph $\Gamma$ with pages $\Gamma_1, \dots , \Gamma_l$ where the gluing spine $\Gamma_0$ consists of a finite set of vertices. 

Let $o \in \Gamma_0$ be a fixed point. (Since the spine is finite, all points belonging to it are essentially the same.) Set $|x| : = \max\{1, d(x, \Gamma_0)\}$. Let $i_x$ denote the index $i$ of the page that $x$ belongs to and $V_{i_x}(y,r)$ denote the volume of the ball of radius $r$ centered at $y$ taken in the page $\Gamma_{i_x}$. If $x$ is in the spine, set $i_x = 0$ and $V_{i_x} = V_{\min}.$ 

\begin{cor}\label{graphs_finite_case}
Assume the hypotheses of Theorem \ref{main_general_thm}. In addition, assume the gluing spine $\Gamma_0$ is finite. Then for all $x,y \in \Gamma$ and sufficiently large $n,$ 

\begin{align}
\begin{split}\label{finite_est}
p(n,x,y) \approx& \frac{C \pi(y)}{V_{i_x}(x, \sqrt{n})} \exp\Big(-\frac{d_{i_{x}}^2(x,y)}{cn}\Big) \\[1ex]
&+C \Bigg[ \frac{|x|^2 |y|^2}{V_{\min}(o, \sqrt{n}) V_{i_x}(x, |x|) V_{j_y}(y, |y|)} 
+ \frac{|x|^2}{V_{j_y}(y, \sqrt{n}) V_{i_x}(x, |x|)}\\[1ex]
&\hphantom{sp}+\frac{|y|^2}{V_{i_x}(x, \sqrt{n}) V_{j_y}(y, |y|)} \Bigg]\exp\Big( - \frac{(|x|^2 + |y|^2)}{cn}\Big).
\end{split}
\end{align}
\end{cor}

In the case $x, \ y$ do not belong to the same page, we take the convention that $d_{i_x}(x,y) = +\infty$, in which case the first term of (\ref{finite_est}) disappears. In this case, we also have $|x| + |y| \approx d_\Gamma(x,y)$. In the case $x, \ y$ are in the same page, then $d_{i_x}(x,y) = d_\Gamma(x,y)$ and $|x|+|y|$ is essentially the distance between $x$ and $y$ requiring a visit to the spine $\Gamma_0.$  

The corollary is an immediate consequence of Theorem \ref{main_general_thm}, with significant simplifications coming from the fact that $\Gamma_0$ is finite, so all points of $\Gamma_0$ can be treated as identical and all sums can be taken over all points of $\Gamma_0.$ The constants contain dependencies on $\Gamma_0,$ such as its diameter and volume. In this case, it is straightforward to deduce the spine-to-spine heat kernel estimates (\ref{central_est}) from the more general (\ref{main_est}) (and to verify that (\ref{main_est}) does not depend upon the choice of page for vertices in the spine).

For instance, if $x \in \Gamma_0,$ then $|x| = 1$ and $V_{i_x}(x, |x|) \approx V_{\min}(o, 1),$ a constant. In the situation both $x,y \in \Gamma_0,$ the terms in the corollary reduce to
\begin{align*}
p(n,x,y) \approx \frac{C}{V_{\min}(o, \sqrt{n})} \exp\Big( - \frac{\text{diam}(\Gamma_0)^2}{cn}\Big),
\end{align*}
as expected from the spine-to-spine heat kernel estimates. 

The estimates in the corollary are the discrete analog of heat kernel estimates in the case of gluing transient manifolds over compact sets found in the work of Grigor'yan and Saloff-Coste; see Theorems 4.9 and 5.10 of \cite{lsc_ag_ends}. In order to see these estimates are the same, recall that in the first term, $V_{i_x}(x, \sqrt{n})$ can be replaced by either $V_{i_y}(y, \sqrt{n})$ or $\sqrt{V_{i_x}(x, \sqrt{n}) V_{i_y}(y, \sqrt{n})}$ by changing the constant in the exponential. Due to our assumption that the volume in all pages grows at least as fast as $r^{2+\epsilon}$, the terms $H(x,t)$ appearing in \cite{lsc_ag_ends} can be approximated as $\frac{|x|^2}{V_{i_x}(x,|x|)}$. The final apparent difference is that in \cite{lsc_ag_ends}, terms such as $V_{j_y}(y, \sqrt{n})$ are instead written as $V_{j_y}(o_{j_y}, \sqrt{n})$ (where $o_{j_y}$ is a closest point to the $o$ in the page containing $y$). However, due to the volume doubling property and the fact that $|x| = \max \{1, d(o, y)\}$, it is always possible to replace one such volume by the other at the cost of changing the constants in and out of the exponential. In Corollary \ref{graphs_finite_case}, we leave the volumes centered at $x$ and $y$ because in the case the spine is \emph{not} finite, the replacement mentioned in the previous sentence can no longer be made. 

The assumptions made in \cite{lsc_ag_ends} are that each end is Harnack and transient (in the classical sense). Let us compare these to our hypotheses here. 

If $\Gamma_0$ is finite, then hypothesis (B1), that the graph be book-like, is automatically satisfied. Hypothesis (B2), that the pages (ends) are Harnack is the same. Hypothesis (B3) is that each page is inner uniform and uniformly $S$-transient in its augmented version. Since $\Gamma_0$ is finite, $\widehat{\Gamma}_i$ is transient and Harnack, and $\widehat{\Gamma}_i \setminus \Gamma_0$ is connected for all $1 \leq i \leq l,$ uniform $S$-transience, $S$-transience, and classical transience are the same. (In  \cite{ed_lsc_trans}, we stated that if $K$ is finite, then $S$-transience and classical transience are the same. We should have been more cautious, for this statement need not be true in the most general sense and our definition in \cite{ed_lsc_trans} relies heavily on the assumption that $\widehat{\Gamma}_i \setminus \Gamma_0$ is connected. However, in the case $\widehat{\Gamma}_i$ is transient and Harnack, as we are interested in here, there is no difficulty.)

However, the statement that each page be inner uniform in its augmented version as well as hypothesis (B4) that each page be uniform in the whole graph are rather subtle. First, removing a finite set from an arbitrary Harnack graph need not preserve uniformity. This is related to a property called relatively connected annuli (RCA), which is in turn related to transience. For instance, RCA holds on Harnack spaces satisfying the volume growth condition (\ref{vol_growth_cond}) \cite[Lemma 6.9]{anthony_bdryHarnack3g} (see also \cite[Theorem 1.6]{mathav_jaschek}, which also includes a capacity condition). In light of the these references, the question of determining whether a transient Harnack graph minus a finite set is automatically (inner) uniform seems a delicate one. Instead of taking a possibly lengthy detour to answer it, we instead reassure the reader with the following: While the hypotheses made here and those of \cite{lsc_ag_ends} may not be precisely the same, they are close. Moreover, if we only wished to consider finite spines as in Corollary \ref{graphs_finite_case}, we could disregard the uniformity hypotheses, and, in places where those are used, make different arguments that appeal to the finiteness of the spine. However, since our main interest in this paper is to allow for infinite spines, the uniformity hypotheses are convenient.

The results of \cite{lsc_ag_ends} also cover the case where some ends can be recurrent, provided at least end one is transient. Such cases are handled by the use of an appropriate $h$ transform to return to the setting where all ends are transient. Giving a general description of the case where only one page need be uniformly $S$-transient would be challenging due to the existence of sets that are $S$-transient but not uniformly so. There may not exist an $h$ transform capable of turning the graph into a uniformly $S$-transient one, and, indeed, the differences between (uniform) $S$-transience and classical transience means one would have to think more carefully about the existence of an appropriate $h$-transform. However, in the case the spine is finite, such difficulties do not hold and the general $h$ transform technique may be carried out. We do so in Examples \ref{Z_D_tail} and \ref{Z_D_plane}.

\section{Example: Gluing pages \texorpdfstring{$\Z^{D_i}$}{that are lattices} along a spine \texorpdfstring{$\Z^k$}{that is a lower-dimensional lattice}}\label{proto_ex}

The prototypical example of a ``book-like'' graph is that of gluing pages which are lattices over a spine of a lower-dimensional lattice. That is, $(\Gamma, \mu,\pi)$ is a book-like graph made up of pages $\Gamma_1 = \Z^{D_1}, \dots, \Gamma_l = \Z^{D_l}$ with a spine $\Gamma_0 = \Z^k$ aligning with the first $k$-coordinates of each page. We take the lazy simple random walk on all graphs. 

Let $D_{\min} = \min_{1 \leq i \leq l } D_i.$ For any $x \in \Gamma,$ let $D_x$ denote the dimension of the page $x$ belongs to, with $i_x$ denoting the index of that page. Set $|x| := \max\{1, d(x, \Gamma_0)\},$ and let $d_{+}(x,y)$ denote the minimum distance between $x$ and $y$ where the path must pass through the gluing spine $\Gamma_0.$ (If $x$ belongs to the spine, one may choose $i_x$ to be anything.) 

We have the following concrete heat kernel estimates for this example. 

\begin{cor}\label{lattice_lattice_hkest}
Let $(\Gamma, \mu, \pi)$ be the book-like graph with pages $\Z^{D_1}, \dots, \Z^{D_l}$ and spine $\Gamma_0 = \Z^k$ as described above. Assume $D_{\min} \geq k+3.$ 

Then for all $x, y \in \Gamma$ and $n \gg d(x, \Gamma_0) + d(y, \Gamma_0) + 1,$
\begin{align}\label{lattice_HK_est}
\begin{split}
p(n,x,y) \approx &\frac{C}{n^{\frac{D_x}{2}}}\exp \Big(-\frac{d_{i_x}^2(x,y)}{cn}\Big) 
+ C\Bigg[ \frac{1}{n^{\frac{D_{\min}}{2}} \, |x|^{D_x-k-2} \, |y|^{D_y-k-2}} \\[1ex]
&+ \frac{1}{n^{\frac{D_y}{2}} \, |x|^{D_x - k - 2}}
+  \frac{1}{n^{\frac{D_x}{2}} \, |y|^{D_y -k -2}}  \Bigg] \exp\Big(-\frac{d_+^2(x,y)}{cn}\Big).
\end{split}
\end{align}
\end{cor} 

Corollary \ref{lattice_lattice_hkest} follows from applying Theorem \ref{main_general_thm} to this example. The details of doing so are fairly complex and are given in the proof below. 

As in the previous section, it is also straightforward to verify (\ref{lattice_HK_est}) reduces to the simpler spine-to-spine estimate (\ref{central_est}). 

\begin{proof}
Many quantities appearing in Theorem \ref{main_general_thm} can be easily computed in this setting. We have $V_{i}(x,r) \approx r^{D_i}$ for all $x \in \Gamma_i, \ 1 \leq i \leq l,$ and $V_{\min}(v, r) \approx r^{D_{\min}}$ for all $v \in \Gamma_0.$ Further, since the global weight function $\pi$ is uniformly bounded, we may treat all appearances of it as a constant. As we saw in the proof of Theorem \ref{main_general_thm}, the uniformity hypothesis (B4) allows distances in any page to be replaced with the global distance in $\Gamma$. 

As in the proof of Theorem \ref{main_general_thm}, we may reduce to the situation that the vertices $x, \ y$ lie away from the spine. Again, we have two cases. 

\noindent \emph{Case 1: $x,\ y$ are in distinct pages and away from the gluing spine $\Gamma_0$}

This is the main case. In this particular setting, since the boundary of any pages is a copy of $\Z^k,$ $P(x,y,n) = \{(v,w) \in \Z^k \times \Z^k: d(x,v) \leq n, d(y,w) \leq n, d(v,w) \leq n \}.$ By estimate (\ref{main_est}) from Theorem \ref{main_general_thm}, 
\begin{align}\label{lattice_firststep}
\begin{split}p(n,x,y) \approx \sum_{(v,w) \in P_{x,y,n}} C \Bigg[&
\frac{1}{n^{\frac{D_{\min}}{2}} \, d(x,v)^{D_x-2}\, d(y,w)^{D_y-2}} \\
&+ \frac{1}{n^{\frac{D_y}{2}} \, d(x,v)^{D_x-2}\, d(v,w)^{D_{\min}-2}}\\
&+ \frac{1}{n^{\frac{D_x}{2}} \, d(y,w)^{D_y-2} \, d(v,w)^{D_{\min}-2}} 
\Bigg]\exp\bigg(-\frac{d_*^2}{cn}\bigg),\end{split}
\end{align}
where $d_*^2 = d^2(x,v)+d^2(v,w)+d^2(w,y).$

Assume $x \in \Gamma_i \setminus [\Gamma_0]_4$ and $v \in \partial \Gamma_i.$ We claim that $d(x,v) \approx |x| + d(v_x,v),$ where we recall $|x| := \max \{1 , d(x, \Gamma_0)\}$ and where $v_x$ is a vertex in $\partial \Gamma_i$ achieving $\min_{v \in \partial \Gamma_i} d(x, v) = d(x, \Gamma_0)$ (if there are multiple such vertices, pick one). In this particular case, $|x| = d(x, \Gamma_0)$. The upper bound follows from the triangle inequality, as $d(x,v) \leq d(x, v_x) + d(v_x, v) = |x| + d(v_x, v).$ The lower bound follows from the fact that $|x| = d(x, v_x) \leq d(x,v)$ by definition of $v_x,$ and also $d(v_x, v) \leq d(v_x, x) + d(x, v) \leq 2 d(x,v),$ using the triangle inequality and the definition of $v_x$ again. Hence (\ref{lattice_firststep}) becomes 
\begin{align}\label{lattice_secondstep}
\begin{split}
p(n,x,y) \approx \sum_{(v,w) \in P_{x,y,n}} C&\Bigg[ 
\frac{1}{n^{\frac{D_{\min}}{2}} \, \big[|x| + d(v_x, v)\big]^{D_x-2}\, \big[|y|+d(w_y, w)\big]^{D_y-2}}\\ 
&+ \frac{1}{n^{\frac{D_y}{2}} \, \big[|x|+d(v_x,v)\big]^{D_x-2}\, d(v,w)^{D_{\min}-2}}\\
&+ \frac{1}{n^{\frac{D_x}{2}} \, \big[|y|+d(w_y,w)\big]^{D_y-2} \, d(v,w)^{D_{\min}-2}} 
\Bigg]
\cdot \exp\bigg(-\frac{d_*^2}{cn}\bigg).
\end{split}
\end{align}
In computing the above double sum, different arguments are needed for the upper and lower bound. 

\noindent \underline{\emph{Upper bound:}} The main idea in the upper bound is that restrictions on the set $P_{x,y,n}$ can be ignored, and we can instead consider the sums over $(v,w) \in \Z^k \times \Z^k$. Additionally, due to the triangle inequality,
\[ d(x,y) \leq d(x,v) + d(v,w) + d(w,y),\]
and, since $(a+b)^2 \approx a^2 + b^2,$ the term $d_*^2$ in the exponential can be replaced with $d^2(x,y)$.

We collect some calculus facts in the following lemma. 
\begin{lemma}\label{ub_sums}
Assume $a \geq 1$ and $A,D\geq k+3.$ Let $d(\cdot, \cdot)$ denote the distance between two points in the lattice $\Z^k$ (or in some fattened version of the lattice $\Z^K$). Then there exists constants $c_1, c_2, c_3 >0$ (independent of $a$) such that
\begin{align}\label{first_sum}
&\sum_{v' \in \Z^k} \frac{1}{[a + d(v^*, v')]^{D-2}} \leq \frac{c_1}{a^{D-k-2}}, \qquad \forall v^* \in \Z^k \\\label{second_sum}
&\sum_{w \in \Z^k} \frac{1}{[1+d(v,w)]^{D-2}} \leq c_2, \qquad \forall v \in \Z^k\\\label{third_sum}
&\sum_{v \in \Z^k} \frac{1}{(a + d(v', v))^{D-2}} \sum_{w \in \Z^k} \frac{1}{[1+d(v,w)]^{A-2}} \leq \frac{c_3}{a^{D-k-2}}.
\end{align} 
\end{lemma} 

\begin{proof}[Proof of Lemma \ref{ub_sums}]

Since the case $k=0$ was covered by the previous section, we may assume $k \geq 1$.

We first prove (\ref{first_sum}). Arrange $\Z^k$ so that $v^*$ is the origin and $d(v^*, v') = |v_1'| + \cdots + |v_k'|,$ where $v' = (v_1', \dots, v_k').$ Then
\begin{align}\label{k_dim_comp} 
\sum_{v' \in \Z^k} \frac{1}{[a + d(v^*, v')]^{D-2}} = \sum_{v_1' \in \Z} \cdots \sum_{v_k' \in \Z} \frac{1}{\big[a + |v_1'| + \cdots + |v_k'\big|]^{D-2}}.
\end{align} 

As we sum in each coordinate, the other coordinates are fixed, so it suffices to consider a one-dimensional sum. In that case, 
\begin{align*}
\sum_{v' \in \Z} \frac{1}{[a + |v'|]^{D-2}} &= 2 \sum_{v' = 0}^\infty \frac{1}{[a+ v']^{D-2}}
\leq \frac{2}{a^{D-2}} + \int_0^\infty \frac{2 dx}{(a+x)^{D-2}} \\[1ex]
&\leq \frac{2}{a^{D-2}} + \int_a^\infty \frac{2 du}{u^{D-2}} 
= \frac{2}{a^{D-2}} + \Big[ -\frac{2}{(D-3)u^{D-3}}\Big]_a^\infty \\[1ex]
&\leq \frac{2}{a^{D-2}} + \frac{2}{(D-3)a^{D-3}} \leq \frac{c}{a^{D-3}},
\end{align*}
where we have used $a \geq 1$ and $c$ is a constant (that depends on $D$). 

Using this computation repeatedly on equation (\ref{k_dim_comp}) yields the desired estimate (\ref{first_sum}):
\begin{align*}
\sum_{v' \in \Z^k} \frac{1}{[a + d(v^*, v')]^{D-2}}  &\leq
\sum_{v_1' \in \Z} \cdots \sum_{v_{k-1}' \in \Z} \frac{c}{[a + |v_1'| + \cdots + |v_{k-1}'|]^{D-3}} \leq \frac{c_1}{a^{D-k-2}}.
\end{align*} 

We now prove (\ref{second_sum}). The copies of $\Z^k$ that $v, w$ belong to need not be exactly the same, but for any $v \in \Z^k,$ there exists $w_v$ in the the ``$w$'' copy of $\Z^k$ that achieves the minimum distance between $v$ and that copy. We have $d(v,w) \approx d(v, w_v) + d(w_v, w)$ and we know $0\leq d(v, w_v) \leq \delta$ by the book-like graph assumption (B1). The ``$1+d(v,w)$'' in (\ref{second_sum}) accounts for the situation that $v=w_v.$ Then the desired result follows by making the denominator smaller and using (\ref{first_sum}):
\begin{align*}
\sum_{w \in \Z^k} \frac{1}{[1+d(v,w)]^{D-2}} &\leq \sum_{w \in \Z^k} \frac{1}{[1+d(w_v, w)]^{D-2}} \leq c_2.
\end{align*}

The inequality (\ref{third_sum}) follows by applying (\ref{second_sum}) to the inner sum, then applying (\ref{first_sum}) to the outer sum. 
\end{proof}

By Lemma \ref{notsens_sum}, since $x$ and $y$ belong to different pages, we may assume that sums are arranged such that $v \not = w$. Then applying Lemma \ref{ub_sums} to (\ref{lattice_secondstep}) gives
\begin{align}
\begin{split}\label{lattice_upper_nice}
p(n,x,y) \leq C\Bigg[&\frac{1}{n^{\frac{D_{\min}}{2}} \,  |x|^{D_x-k-2}\, |y|^{D_y-k-2}}
+ \frac{1}{n^{\frac{D_y}{2}} \, |x|^{D_x-k-2}}
\\&+ \frac{1}{n^{\frac{D_x}{2}} \, |y|^{D_y-k-2}} \Bigg]
 \exp\bigg(-\frac{d^2(x,y)}{cn}\bigg).
 \end{split}
\end{align}

In this case, $d(x,y) = d_+(x,y)$ and $d_{i_x}(x,y) = +\infty,$ so the first term of (\ref{lattice_HK_est}) vanishes and that estimate matches (\ref{lattice_upper_nice}).

\noindent \underline{\emph{Lower bound:}} For the lower bound, a key idea is to throw away any terms in the sum that do not count. We will see that only looking at terms in a certain ``window'' produces a lower bound matching the upper bound found above. In particular, we only consider $v,\ w$ in the gluing spine at distance on the scale of $d(x,y)$ away from $x, \ y,$ respectively (or closer).

The full sum is over the set $P_{x,y, \lfloor \frac{n}{4}\rfloor}$ since this is the lower bound. We only consider a subset of this set, which we call a ``window.'' Given $x,\ y$ in distinct pages, let $v_x,\ w_y$ denote a choice of closest points in $\Gamma_0$ to $x, y,$ respectively. Select a geodesic path between $x$ and $y.$ Since $x,\ y$ are in distinct pages, such a geodesic path must cross both $\partial \Gamma_{i_x}$ and $\partial \Gamma_{i_y},$ so we may choose vertices $v_g \in \partial \Gamma_{i_x},\ w_g \in \partial \Gamma_{i_y}$ such that $d(x,y) = d(x, v_g) + d(v_g, w_g) + d(w_g, y).$ (We use the subscript $g$ to allude to ``geodesic''.) 

Define $W_{x, y, K} := \{ v \in \Gamma_{i_x} = \Z^k : d(v,v_g) \leq K d(x,y)\},$ where $K$ is a fixed constant. We will often consider $K=4$. The set $W_{y,x,K}$ is defined analogously.

The set $W_{x,y,4}$ has the property that it contains $v_x$ since 
\[ d(v_x, v_g) \leq d(v_g,x)+d(x, v_x) \leq 2 d(x,y).\]

Moreover, if $v \in W_{x,y,4}$ then 
\[ d(x,v) \leq d(x, v_x) + d(v_x,v)\leq d(x,y) + 4d(x,y) \leq 5 d(x,y).\]
If $v \in W_{x,y,4}$ and $w \in W_{y,x,4},$ then 
\[ d(v,w) \leq d(v, v_g) + d(v_g, w_g) + d(w_g, w) \leq 9 d(x,y).\]

Therefore, provided
$\lfloor \frac{n}{4}\rfloor \geq 9d(x,y)$ it follows that $ W_{x,y,4} \times W_{y,x,4} \subseteq P_{x, y \lfloor \frac{n}{4} \rfloor}$. (This is one place where the assumption $n \gg d(x,y)$ in the theorem is relevant.) 

On the set $W_{x,y,4} \times W_{y,x,4}$, all of the distances $d(x,v), d(v,w), d(y,w)$ are controlled by $d(x,y),$ so we may again replace the sum of distances $d_*^2$ in the exponential by $d^2(x,y)$.

We now need an analog of Lemma \ref{ub_sums} for the lower bound.

\begin{lemma}\label{lb_sums} 
Let $\widehat{K}, K$ be fixed constants and $W_{x,y,K}$ be defined as above. Assume $0\leq a \leq \widehat{K}d(x,y)$ and $D \geq 3+k.$  Then there exists a constant $c_4 >0$ independent of $a$ such that 
\begin{equation}
\sum_{v \in \Z^k: d(v,v') \leq K d(x,y)} \frac{1}{[a+d(v',v)]^{D-2}} \geq \frac{c_4}{a^{D-k-2}}, \quad \forall v' \in \Z^k.
\end{equation}
\end{lemma}

\begin{proof}
Arrange the coordinate system on the $\Z^k$ so that $v'$ is the origin and  $d(v, v') = |v_1| + \cdots + |v_k| \leq Kd(x,y).$ Therefore $|v_i| \leq K d(x,y)$ for all $1 \leq i \leq k.$ The lemma follows if we can prove the one-dimensional lower bound
\begin{equation}\label{1D_lowersum}
\sum_{v=0}^{K d(x,y)} \frac{1}{[a+v]^{D-2}} \geq \frac{c}{a^{D-3}}
\end{equation}
for some constant $c>0.$
We may need to slightly adjust $\widehat{K}, K$ at each repetition of the one-dimensional sum. 

There are two cases to consider in order to prove (\ref{1D_lowersum}). 

\emph{(*)The case $a \leq Kd(x,y):$}

In this case,
\begin{align*}
\sum_{v=0}^{Kd(x,y)} \frac{1}{(a+|v|)^{D-2}}
&= \sum_{v=0}^a \frac{1}{(a+|v|)^{D-2}} + \sum_{v=a+1}^{Kd(x,y)} \frac{1}{(a+|v|)^{D-2}} 
 \geq \sum_{v=0}^a \frac{1}{(a+|v|)^{D-2}} \\[1ex]
&\geq \sum_{v=0}^a \frac{1}{2^{D-2}\, a^{D-2}} = \frac{c_4 a}{a^{D-2}}
= \frac{c_4}{a^{D-3}}. 
\end{align*} 

\emph{(**) The case $Kd(x,y) \leq a \leq \widehat{K} d(x,y):$} 

In this case $a \approx d(x,y)$ and $\tilde{k}a = Kd(x,y)$ for some value of $\tilde{k}$ (which is bounded above/below). Thus, since $a \leq a + |v| \leq (K+1)a,$
\begin{align*}
\sum_{v=0}^{Kd(x,y)} \frac{1}{(a+|v|)^{D-2}} \geq \sum_{v=0}^{\tilde{k}a} \frac{c}{a^{D-2}} = \frac{c_4}{a^{D-3}}.
\end{align*} 
\end{proof}

Using the lemma, we estimate the first term of  (\ref{lattice_secondstep}) as
\begin{align*}
{\color{cb_blue}{\text{TERM } A}} &\approx 
\sum_{(v,w) \in P_{x,y,\lfloor \frac{n}{4}\rfloor}}
\frac{C}{n^{\frac{D_{\min}}{2}} \, \big[|x| + d(v_x, v)\big]^{D_x-2}\, \big[|y|+d(w_y, w)\big]^{D_y-2}} \\
& \geq \sum_{(v,w) \in W_{x,y,4} \times W_{y,x,4}} \frac{C}{n^{\frac{D_{\min}}{2}} \, \big[|x| + d(v_x, v)\big]^{D_x-2}\, \big[|y|+d(w_y, w)\big]^{D_y-2}} \\
&\geq \frac{C}{n^{\frac{D_{\min}}{2}}|x|^{D_x-k-2}|y|^{D_y-k-2}}.
\end{align*}

To deal with the terms coming from the initial ${\color{cb_teal}{\text{TERM } B_x}}$ and ${\color{cb_magenta}{\text{TERM } B_y}}$, we look at slightly different windows. For $v \in W_{x,y,4},$ let $w_v$ denote a closest point in $\partial \Gamma_{i_y}$ to $v$. Since $v$ is in the spine and $w_v$ is a point also in the spine but close to a specific page, the $\delta$-book-like property of the spine implies $d(v,w_v) \leq \delta$. Then 
\begin{align*}
d(w_v, w_g) &\leq d(w_v, v) + d(v, v_g) + d(v_g, w_g) \leq \delta +4d(x,y) + d(x,y) \\&\leq (5+\delta) d(x,y),
\end{align*} 
where $v_g, \ w_g$ are point on a shortest (geodesic) path between $x$ and $y$ as seen in the definition of $W_{x,y,C}$. Therefore $w_v \in W_{y,x, 5+\delta}$. Again assuming $n \gg d(x,y)$, then $\{(v,w) \in W_{x,y,4} \times W_{y,x, 5+\delta}\} \subseteq P_{x,y, \lfloor \frac{n}{4}\rfloor}$. Then 
\begin{align*}
{\color{cb_teal}{\text{TERM } B_x}} &\approx \sum_{(v,w) \in P_{x,y, \lfloor \frac{n}{4}\rfloor}} \frac{1}{n^{\frac{D_y}{2}}[|x|+d(v_x,v)]^{D_x-2} \, d(v,w)^{D_{\min}-2}} \\
& \geq \frac{1}{n^{\frac{D_y}{2}}} \sum_{v \in W_{x,y,4}} \frac{1}{[|x|+d(v_x,v)]^{D_x-2}} \sum_{w \in W_{y,x,5+\delta}} \frac{1}{d(v,w)^{D_{\min}-2}}
\\
&\geq \frac{1}{n^{\frac{D_y}{2}}} \sum_{v \in W_{x,y,4}} \frac{1}{[|x|+d(v_x,v)]^{D_x-2}} \frac{1}{d(v,w_v)^{D_{\min}-2}}\\
 &\geq \frac{1}{n^{\frac{D_y}{2}}\, \delta^{D_{\min}-2} \, |x|^{D_x-k-2}} = \frac{1}{n^{\frac{D_y}{2}}\,|x|^{D_x-k-2}},
\end{align*}
Here we have first reduced the sum over $P_{x,y, \lfloor \frac{n}{4}\rfloor}$ to only be over the given windows, then used the fact that $w_v \in W_{y,x,5+\delta}$ for all $v \in W_{x,y,4},$ and finally we used the $\delta$-book-like property of the spine along with Lemma \ref{lb_sums}.

The same line of reasoning can be applied to the last term of (\ref{lattice_secondstep}):
\begin{align*}
{\color{cb_magenta}{\text{TERM } B_y}} &\approx     \frac{C}{n^{\frac{D_x}{2}} \, \big[|y|+d(w_y,w)\big]^{D_y-2} \, d(v,w)^{D_{\min}-2}} \\
&\geq \frac{C}{n^{\frac{D_x}{2}} \, |y|^{D_y-k-2}}.
\end{align*}
Putting the three terms together gives a lower bound that is of exactly the same form as the upper bound (\ref{lattice_upper_nice}), which matches the estimates in the corollary. 

\noindent \emph{Case 2: $x, \ y$ are in the same page away from the gluing spine $\Gamma_0$}

The arguments are very similar to that of the distinct pages case. Applying estimate (\ref{main_est}) from Theorem \ref{main_general_thm} and recalling $D_x=D_y$ gives 
\begin{align*}
p(n,x,y) &\approx
\frac{C}{n^{\frac{D_x}{2}}} \exp\Big(-\frac{d_{i_x}^2(x,y)}{cn}\Big) 
+\sum_{(v,w) \in P_{x,y,\lfloor \frac{n}{4}\rfloor}}
C\Bigg[ \frac{1}{n^{\frac{D_{\min}}{2}}\, d(x,v)^{i_x-2}\, d(y,w)^{i_x-2}}\\
&+ \frac{1}{n^{\frac{D_x}{2}} \, d(x,v)^{i_x-2} \, d(v,w)^{D_{\min}-2}} 
+ \frac{1}{n^{\frac{D_x}{2}} \, d(y,w)^{i_x-2} \, d(v,w)^{D_{\min}-2}} \Bigg]\\[1ex]
&\cdot \exp\bigg(-\frac{[d^2(x,v)+d^2(v,w)+d^2(y,w)]}{cn}\bigg).
\end{align*}

The first term is precisely the first term of (\ref{lattice_HK_est}). For the upper bound, it is obvious that the numerator of the exponential is controlled by $d_+(x,y),$ since this quantity can be thought of as minimizing over all paths between $x$ and $y$ that hit a $v$ and a $w$ (it could be that $v=w$). The desired upper bound follows by repeating the same arguments as in Case 1 to compute the sums.

For the lower bound, we again use the idea of only looking at the terms in a certain ``window.'' Take a path between $x$ and $y$ that hits $\Gamma_0$ and achieves $d_+(x,y).$ Then $d_+(x,y) = d(x,v_g) + d(v_g,y)$ for some $v_g \in \Gamma_0.$ Then define $W^+_{x,y,K} := \{ v \in \partial \Gamma_{i_x} : d(v,v_g) \leq K d_+(x,y)$. Note $|x|, \ |y|$ are controlled by $d_+$. Repeating the lower bound arguments of Case 1 using the windows $W^+$ gives a bound of the same form.   
\end{proof}

\begin{remark}
Consider gluing a finite number of lattices of dimension at least three via a point. One continuous analog of this is that of manifolds glued via a compact set, which is described in Grigor'yan and Saloff-Coste in \cite{lsc_ag_ends}. In \cite{ooi_bmvd}, large time estimates for $\R^d$ and $\R^{d'}$ glued via a compact set (point) via Dirichlet forms are given in Theorem 1.7. Moreover, gluing a finite number of lattices of dimension at least four via a line can be thought of as multiplying the situation of the previous sentence by an extra dimension (and so on). The results of Corollary \ref{lattice_lattice_hkest} match these results in the continuous case (see Sections 1.2 and 4.5 of \cite{lsc_ag_ends} and Theorem 1.1 of \cite{ag_ishiwata} with $\alpha=0$).
\end{remark}

\begin{remark}
In addition to matching the results expected from \cite{lsc_ag_ends, ag_ishiwata, ooi_bmvd}, the results of Corollary \ref{lattice_lattice_hkest} are sufficiently robust as to allow for certain perturbations which destroy the reduction of dimension argument. For instance, our corollary extends to the case where the pages and spine are only quasi-isometric to lattices. The spine $\Gamma_0$ also need not have such precise symmetry. Below we describe a more general gluing of not-quite lattices for which Corollary \ref{lattice_lattice_hkest} still holds. 

As above, let $k$ be the dimension of the gluing spine and $D_1, \dots, D_l$ be positive integers greater than or equal to $k+3$ be the dimensions of $l$ pages. Let $M = \max \{D_1, \dots, D_l\}.$ Consider pages $\Gamma_{1}, \dots, \Gamma_l$ such that $\Gamma_i$ is quasi-isometric to $\Z^{D_i}$ for all $1\leq i \leq l$ as described below.

\begin{definition}[Quasi-isometry]\label{quasi-iso} Let $(M_1, d_1, \pi_1)$ and $(M_2, d_2, \pi_2)$ be two metric measure spaces. We say $M_1$ and $M_2$ are \emph{quasi-isometric} if there exists a function $\Phi : M_1 \to M_2$ such that 
\begin{enumerate}[(1)]
\item There exists $\varepsilon > 0$ such that the $\varepsilon$-neighborhood of the image of $\Phi$ is equal to $M_2.$
\item There exist constants $a, b$ such that 
\[ a^{-1} d_1(x,y) - b \leq d_2(\Phi(x), \Phi(y)) \leq a d_1(x,y)+b \quad \forall \ x,y \in M_1.\]
\item There exists a constant $C_q >0$ such that 
\[ \frac{1}{C_q} \pi_1(x) \leq \pi_2(\Phi(x)) \leq C_q \pi_1(x).\]
\end{enumerate}
We call such a function $\Phi$ a \emph{quasi-isometry}.
\end{definition}

Defining the gluing spine in this situation is more subtle. In $\R^{M},$ take a (possibly affine) subspace of dimension $k;$ call it $S.$ Fix $\varepsilon_s,\ \varepsilon_r >0.$ To construct $\Gamma_0,$ proceed as follows. First take the $\varepsilon_s$ neighborhood of $S, \ N(S, \varepsilon_s).$ Then take a discrete subset $\Gamma_0$ of points in $N(S, \varepsilon_s).$ Turn $\Gamma_0$ into a graph (also denoted $\Gamma_0$) by connecting two vertices (elements) in $\Gamma_0$ via an edge if and only if they are at distance less than $\varepsilon_r$ from each other in $\R^M.$ We require $\Gamma_0$ to be such that, as a graph, it is quasi-isometric to $N(S, \varepsilon_s).$

As in Section \ref{glue_setup}, to describe the gluing, we need identification maps (bijections) between subsets of $\Gamma_0$ and $\Gamma_i.$ We will require all such subsets to also be quasi-isometric to $N(S, \varepsilon_S)$. 

Although the above description is somewhat complicated, the proof of Corollary \ref{lattice_lattice_hkest} remains essentially unchanged. On pages that are quasi-isometric to square lattices, volumes and distances only change up to constants (depending on the quasi-isometry constants). From Lemma \ref{notsens_sum}, the computation of sums over $\Z^k$ or subsets thereof has sufficient flexibility that summing over a spine only quasi-isometric to $\Z^k$ does not change the general character. (Indeed, as seen in the lower bound arguments, the important terms all come from only a small part of the sum. In this new quasi-isometric setting, there must still be sufficient points in the important windows for the computations to go through (modulo changes in constants, which we did not track in the first place). 
\end{remark}

\section{Further examples}\label{morebooklike_ex} 

In this section, we consider further applications of Theorem \ref{main_general_thm}.

\begin{example}
    Consider a book-like graph with pages $(\Gamma_i, \mu^i, \pi_i), \ 1 \leq i \leq l$. Let $\Gamma_i = \Z^{D_i}$ and $(\mu^i, \pi_i)$ describe the lazy simple random walk on $\Z^{D_i}$. In other words, $\mu^i \equiv 1$ and $\pi_i$ is two times the degree of the lattice $\Z^{D_i}$. Fix an integer $k\geq 0$ and a spine $\Gamma_0 = \Z^k$, and assume $D_{\min} \geq k+3$. We give a way to construct another random walk structure on the pages,

    For each $1 \leq i \leq l,$ fix $\alpha_i >0$. Define $\psi_i(x) = (1+|x|)^{\alpha_i},$ where $|x| = d(x,\Gamma_0).$ Consider the graph $(\Gamma_i, \widetilde{\mu}^i, \widetilde{\pi}_i)$ given by 
    \[ \widetilde{\pi}_i(x) = \pi_i(x) \psi_i(x) , \qquad \widetilde{\mu}^i_{xy} = \mu^i_{xy} \min\{\psi_i(x), \psi_i(y)\}.\]
    These weights are controlled and uniformly lazy, and formula (\ref{MK_formula}) then defines a Markov kernel $\widetilde{K}$ on $\Gamma_i$. This is the well-known Metropolis chain, see for instance Example 6.1 of \cite{pd_khe_lsc_analytic}. A random walk on $(\Gamma_i, \widetilde{\mu}^i, \widetilde{\pi}_i)$ is biased against moving towards the gluing spine $\Gamma_0$.

    Consider gluing these pages with new random walk structures together along $\Gamma_0$. The underlying geometry of the graph is unchanged, so the glued graph is still book-like and the uniformity hypotheses regarding the spine still hold. Each page $(\Gamma_i, \widetilde{\mu}^i, \widetilde{\pi}_i)$ remains Harnack through arguments as in \cite{ag_lsc_harnackstability}; indeed, as seen there, $\alpha_i$ can also be taken to be negative. The spine $\Z^k$ is even more uniformly $S$-transient in each page since the volume of the overall page has gotten bigger (as $\widetilde{\pi}_i \geq \pi_i$ due to $\alpha_i$ positive). Therefore, this example satisfies our hypotheses, and Theorem \ref{main_general_thm} and all of our methods discussed above still apply. However, computations as done in Section \ref{proto_ex} would become more complicated. 
\end{example}

\begin{example}[Gluing a copy of $\Z^4$ and $\Z^5$ via two arbitrary points]\label{lattice_2arb}
In fact, as mentioned in Remark \ref{alpha_conn}, the hypothesis that the spine be $\alpha$-connected is not necessary. Consider the following relatively simple example.

Let $\Gamma_1 = \Z^4$ with lazy simple random walk and pick two points $o_1, o_2 \in \Z^4$. Glue $\Gamma_1$ to a second page $\Gamma_2  = \Z^5$ with lazy simple random walk by identifying the points $o_1 \in \Z^4$ and $(o_1,0) \in \Z^5$, as well as the points $o_2 \in \Z^4$ and $(o_2, 0) \in \Z^5$. (We take the last coordinate in the $\Z^5$ points to be zero for simplicity, but this is not necessary.) This creates a book-like graph with a spine $\Gamma_0 = \{ o_1, o_2\}$.  The spine is finite, but $\alpha := d_\Gamma(o_1, o_2)$ may tend toward infinity if the points $o_1, o_2$ are moved further and further apart.

While we no longer wish to apply the reasoning of Corollary \ref{graphs_finite_case}, which explicitly thinks of the spine as having a fixed and unchanging diameter, the conclusion of this corollary nevertheless remains true. 

For instance, if $x \in \Z^4, \ y \in \Z^5,$ then for all sufficiently large $n,$
\[ p(n,x,y) \approx C\bigg[\frac{1}{n^{5/2}|x|^2} + \frac{1}{n^2 |y|^3 }\bigg] \exp\Big(-\frac{d^2(x,y)}{cn}\Big),\]
where $|x| := \max\{1, d(x, \Gamma_0)\}$. The quantity $\alpha$ does not appear explicitly in this estimate, but it is implicitly present. For instance, if $x, y$ are both closer to $o_1$ than $o_2,$ then in the computation of the sums appearing in Theorem \ref{main_general_thm}, the terms coming from the pair $(o_1, o_1)$ are largest and one can essentially ignore the second point. On the other hand, if say $d(x, \Gamma_0)=o_1$ and $d(y,\Gamma_0) = o_2,$ it still must be that a geodesic path between $x$ and $y$ contains either $o_1$ or $o_2$. Without loss of generality, assume it contains $o_1$ so that $d(x,y) = d(x, o_1) + d(o_1,y).$ Then $\alpha = d(o_1, o_2) \leq d(o_1, x) + d(x, y)+d(y,o_2)\leq 3d(x,y)$, that is, the distance between the two points in the gluing set is controlled by $d(x,y)$ and the terms coming from the pair $(o_1,o_2)$ will dominate.

By similar arguments, if $x,y \in \Z^5$ then for $n$ sufficiently large,
\begin{align*}
p(n,x,y) &\approx \frac{C}{n^{5/2}} \exp\Big(-\frac{d^2(x,y)}{cn}\Big) \\
&+ C \bigg[\frac{1}{n^2|x|^3|y|^3} + \frac{1}{n^{5/2}|x|^2} + \frac{1}{n^{5/2}|y|^3}\bigg] \exp\Big(-\frac{d_+^2(x,y)}{cn}\Big).
\end{align*}
As in the previous case, it may be that the value of $\alpha$ is controlled by $d_+(x,y)$ (again consider a shortest possible path between $x$ and $y$ that visits at least one of $o_1, o_2$). On the other hand, it may be that both $x,y$ are much closer to say $o_1$ than $o_2$, in which case $\alpha$ is somewhat irrelevant as the terms involving $o_2$ may be ignored (as they are dominated by other terms in the sum). 

Similar types of estimates can be given for other locations of $x$ and $y$. Of course, this example can be generalized to more pages of varying dimensions or to gluing over a finite number of points (that may be increasingly far apart).

\end{example}

\begin{example}[Gluing lattices via a sparse line]\label{lattice_sparse}

Consider the situation of the previous example with pages $\Z^4$ and $\Z^5$ with lazy simple random walk. We now wish to consider gluing over an infinite set of points that become sparser. For instance, consider the sparse line given by taking points with first coordinate coordinate $2^n$ for some positive number $n$ and all other coordinates zero. In this case, Theorem \ref{main_general_thm} still applies, and one could still consider the ``windows'' type argument given in Section \ref{proto_ex} to think about what points in the spine should ``count" in the final sum. However, what points ``count" will depend heavily on the location of both $x$ and $y$ as well as on $n$. Moreover, the calculus arguments from Section \ref{proto_ex} that relied upon having a line of points no longer hold. Writing the estimates from Theorem \ref{main_general_thm} in a more readable form would take a non-trivial amount of work and would involve many cases.

Theorem \ref{main_general_thm} will still hold if we glue over sparse lines with less homogeneity in the spread of their points or over sparse higher dimensional lattices (provided the dimension of the pages is sufficiently large). However, it becomes even less tractable to write down any more concrete estimates. 
\end{example}

\begin{example}[Gluing two $\Z^4$'s via a half-line]
Consider two copies of the lattice $\Z^4$ with lazy SRW. Identify the non-negative part of the $x_4$-axis in each copy. In this case, we have two pages, each a copy of $\Z^4,$ and $\Gamma_0$ is a half-line. This glued graph satisfies hypotheses (B1)-(B4) and the additional volume growth condition (\ref{minvol_faster2}) so that Theorem \ref{main_general_thm} applies. In fact, the computations of a more concrete estimate in this setting are essentially the same as those in Corollary \ref{lattice_lattice_hkest}, except that the quantities $|x|, |y|$ appearing in (\ref{lattice_HK_est}) now represent the distance of $x,y$ to the $x_4$-half axis (as opposed to just the $x_4$-axis). This difference also comes into play in computing $d_+.$ While the estimate has the same \emph{form} as (\ref{lattice_HK_est}), the terms appearing are calculated differently. In particular, if $x=(x_1,x_2,x_3,x_4)$ where $x_4$ is large and negative, then $|x|$ is much larger when gluing over the half-line than when gluing over the line. 

This example generalizes to gluing $\Z^{D_1}, \dots, \Z^{D_l}$ via a $k$ dimensional half-lattice by identifying the half-lattice made up of the last $k$ coordinates where the very last coordinate is positive. Again, it is necessary that $\min_{1 \leq i \leq l} D_i \geq k +3$ in order for hypothesis (B3) to hold. Then again the estimate looks like that of Corollary \ref{lattice_lattice_hkest}. 

The above paragraphs also hold if we were to glue lattices along a fattened half-lattice or along a more general set that is quasi-isometric to a half-lattice. 
\end{example}

\begin{example}[Gluing lattices via a two-dimensional cone]
Consider taking a set of points in a copy of $\Z^2$ that corresponds to a cone. For instance, take a cone of aperture $\alpha$ in $\R^2,$ and then take the set of lattice points lying inside of that cone. Consider $l$ lattices $\Z^{D_i}$ with $D_i \geq 5$ for all $1 \leq i \leq l$ and identify the chosen cone points across the last two coordinate axes. This gives a book-like graph with $\Gamma_i = \Z^{D_i}$ and $\Gamma_0$ the chosen cone. Once again, we can apply both Theorem \ref{main_general_thm} and the reasoning of Corollary \ref{lattice_lattice_hkest} to get heat kernel estimates of the form (\ref{lattice_HK_est}), with distances interpreted appropriately. (Note the arguments of Section \ref{proto_ex} show that the overall character of the sums is controlled only by terms in a particular window. In any such window, there are enough points in the cone to make the overall approximation of the sum unchanged.) 

\end{example}

\begin{example}[$\Z^D$ with a half-line tail]\label{Z_D_tail}
Consider gluing a lattice $\Z^D$ with $D \geq 3$ to the half line $\Z_{\geq 0}$ by identifying their origins and taking the lazy simple random walk on each page. The resulting graph $(\Gamma, \mathcal{K}, \pi)$ is book-like, but $\Z_{\geq 0}$ is $S$-recurrent, so hypothesis (B3) is not satisfied. However, by using Doob's $h$-transform, we can turn this example into one that does resemble our hypotheses. This example is an extension of Example 7.8 of \cite{FK_glue}, which used the $h$-transform technique to determine that if $o$ is the origin, then
\[ p(n,o,o) \approx \frac{C}{n^{3/2}}.\]

As in that example, we construct a positive harmonic function $h$ on $\Gamma$. Since $D \geq 3, \ \Z^D$ admits a Green function $G_{\Z^D}(x,o) = \sum_n p_{\Z^D}(n, x,o)$, where $o$ is the origin. The Green function is harmonic on $\Z^d \setminus \{ o\},$ superharmonic on all of $\Z^3,$ has a finite value at the origin, and can be approximated as 
\[ G_{\Z^d}(x,o) \approx \begin{cases}\frac{1}{|x|^{D-3}}, & x \not = o \\[1ex] \sum_{n} \frac{1}{n^{D/2}}, & x=o. \end{cases}\]
Recall $|x|$ is the distance from $x$ to the gluing spine, so in this case $|x| = d(x,o)$. Notice $G_{\Z^d}(x,o) \to 0$ as $|x| \to \infty$. All harmonic functions on the half-line are linear.

Choose $a>0$ such that
\[ h(x) = \begin{cases} 1+ G_{\Z^D}(x,o), & x \in \Z^D \setminus \{o\} \\ 1+ G_{\Z^D}(o,o) , & x =o \\ ax + 1+ G_{\Z^D}(o,o), & x \in \Z_{\geq 0} \end{cases}\]
is a harmonic function on $\Z^D$ with a $Z_{\geq 0}$ tail. Such a value $a>0$ exists since $G_{\Z^D}$ is superharmonic. 

Consider the $h$-transform of $(\Gamma, \mathcal{K}, \pi)$, denoted $(\Gamma_h, \pi_h, \mu_h)$ where
\begin{align*}
    &\Gamma_h = \Gamma \text{ (as graphs)}\\
    &\pi_h(x) := h^2(x)\pi(x) \quad \forall x  \in \Gamma \\
    &\mathcal{K}_h(x,y) = \frac{1}{h(x)} \mathcal{K}(x,y) h(y) \quad \forall x,y \in \Gamma.
\end{align*}
The $h$-transform graph $(\Gamma_h, \pi_h, \mu_h)$ has controlled weights and is uniformly lazy since $(\Gamma, \mathcal{K}, \pi)$ satisfies those hypotheses and $h$ is harmonic. The heat kernel on the original book-like graph is related to the $h$-transform graph via the formula
\[ p_\Gamma(n,x,y) = h(x)h(y) p_{\Gamma_h}(n,x,y).\]
The technique of using $h$-transforms to move from recurrent spaces to transient ones is well-studied. It follows that if $\Gamma$ has Harnack pages, then the same is true of $\Gamma_h,$ and changing the weights does not change (inner) uniformity. It also follows from  Theorem 2.9 of \cite{ed_lsc_trans} that the pages of $\Gamma_h$ are uniformly $S$-transient, since the volume in the $\Z_{\geq 0}$ page now behaves like volume on $\Z^3.$

Since $\Gamma_0$ is finite, Corollary \ref{graphs_finite_case} applies to the graph $(\Gamma_h, \mathcal{K}_h, \pi_h)$. In general, let $V_{i_x, h}$ denote the volume of the page to which $x$-belongs with respect to $\pi_h$. Then: 
\begin{align}
\begin{split}\label{lattice_tail_est}
    p_\Gamma(n,x,y) &= h(x)h(y) p_{\Gamma_h}(n,x,y) \approx \frac{Ch(x)h(y)}{V_{i_x, h}(x, \sqrt{n})} \exp\Big(-\frac{d_{i_{x}}^2(x,y)}{cn}\Big) \\[1ex]
&+Ch(x)h(y)\Bigg[ \frac{|x|^2 |y|^2}{V_{\min, h}(o, \sqrt{n}) V_{i_x, h}(x, |x|) V_{j_y, h}(y, |y|)} 
\\&+ \frac{|x|^2}{V_{j_y, h}(y, \sqrt{n}) V_{i_x, h}(x, |x|)}
+\frac{|y|^2}{V_{i_x, h}(x, \sqrt{n}) V_{j_y, h}(y, |y|)} \Bigg]\\[1ex]&\cdot\exp\Big( - \frac{(|x|^2 + |y|^2)}{cn}\Big).
\end{split}
\end{align}

If $x \in \Z^d$, then $h(x) \approx 1$ and $V_{i_x,h}(x,r) \approx r^d$. If $x \in \Z_{\geq 0},$ then  $h(x) \approx |x|$ and $V_{i_x, h}(x,r) \approx r(|x| + r)^2.$ Therefore $V_{\min, h}(o, \sqrt{n}) \approx n^{3/2}$. Notice $|x| = x$ for points in the half-line tail. 

Based on what pages $x, \ y$ are in, the general estimate (\ref{lattice_tail_est}) simplifies as written below. We have only written the terms that contribute, and in some cases, we have used the volume doubling property combined with a change of constant in the exponential to rewrite terms in a different way. These results match the continuous setting results of Example 6.11 of \cite{lsc_ag_ends} when $d=3$ (manifolds) and of Theorem 1.1 of \cite{ooi_bmvd} (Dirichlet forms). 
\begin{enumerate}[(1)]
\item As seen previously,
\[ p(n,o,o) \approx \frac{C}{n^{3/2}}.\]
\item If $x, y \in \Z_{\geq 0}$ (both in the half-line tail),
\begin{align*}
    p(n,x,y) &\approx \frac{C|x||y|}{\sqrt{n}\, (|x|+\sqrt{n})^2} \exp\Big(-\frac{d^2(x,y)}{cn}\Big)\\
    &\approx \frac{C|x||y|}{\sqrt{n}(|x|+\sqrt{n})(|y|+\sqrt{n})} \exp\Big(-\frac{d^2(x,y)}{cn}\Big) 
\end{align*}
\item If $x,y \in \Z^d$ (both in the lattice)
\begin{align*}
    p(n,x,y) &\approx \frac{C}{n^{D/2}}\exp\Big(-\frac{d^2(x,y)}{cn}\Big)
    \\&+ \frac{C}{n^{3/2} |x|^{D-2} |y|^{D-2}}\exp\Big(-\frac{(|x|^2+|y|^2)}{cn}\Big)
\end{align*}
\item If $x \in \Z_{\geq 0}, \ y \in \Z^d$ (different pages case)
\begin{align*}
    p(n,x,y) &\approx C\Bigg[\frac{1}{n^{D/2}} + \frac{|x|}{\sqrt{n} \, (|x|+\sqrt{n})^2 |y|^{D-2}}\Bigg]\exp\Big(-\frac{d^2(x,y)}{cn}\Big)\\
    &\approx  C\Bigg[\frac{1}{n^{D/2}} + \frac{|x|}{n^{3/2} |y|^{D-2}}\Bigg]\exp\Big(-\frac{(|x|^2+|y|^2)}{cn}\Big)
\end{align*}
\item In the case that $x$ is in one page, while $y=o,$ (or vice versa), then if $x \in \Z_{\geq 0},$ cases (2) and (4) both give 
\[ p(n,x,o) \approx \frac{C|x|}{n^{3/2}} \exp\Big(-\frac{|x|^2}{cn}\Big).\]
If instead $x \in \Z^d$, then cases (3) and (4) both yield
\[ p(n,x,o) \approx C\bigg[ \frac{1}{n^{D/2}} + \frac{1}{n^{3/2}|x|^{D-2}}\bigg]\exp\Big(-\frac{|x|^2}{cn}\Big). \]
\end{enumerate}
\end{example}

\begin{example}[Gluing $\Z^D$ to $\Z^2$ via a point]\label{Z_D_plane}
Consider two pages, $\Gamma_1 = \Z^D$ with $D\geq 3$ and $\Gamma_2 = \Z^2,$ both with lazy SRW. Construct a glued graph $\Gamma$ by identifying the origins of these two lattices; call the single origin $o.$ As in the previous example, the glued graph $(\Gamma, \mathcal{K}, \pi)$ satisfies hypotheses (B1), (B2), and (B4), but hypothesis (B3) is not satisfied since $\Z^2$ is $S$-recurrent with respect to any finite set of vertices. 

Let $g(x) = \sum_{n=0}^\infty [p_{\Z^2}(n,o,o)-p_{\Z^2}(n,o,x)]$. This is a harmonic function on $\Z^2 \setminus \{o\}$, and it is known that $g(x) \approx \log(|x|),$ where $|x|$ denotes the distance from $x$ to the origin $o$. (See e.g. Spitzer \cite{Spitzer} or Lawler and Limic Chapter 6 \cite{lawler_limic}). 

As in the previous example, it is possible to choose $a>0$ such that 
\[ h(x) = \begin{cases} 1+ G_{\Z^D}(x,o), & x \in \Z^D \setminus \{o\} \\ 1+ G_{\Z^D}(o,o), & x =o \\ ag(x)+1+G_{\Z^D}(o,o), & x \in \Z^2 \setminus \{o\}\end{cases}\]
is a harmonic function on the glued graph $\Gamma$.

The $h$-transform $(\Gamma_h, \mathcal{K}_h, \pi_h)$ satisfies hypotheses (B1)-(B4). Note $V_{i_x, h}(x, r) \approx r^2 \log^2(1+|x|+r)$ for $x \in \Z^{2}$. That hypothesis (B3) is now satisfied follows from Theorem 2.9 of \cite{ed_lsc_trans}, and in Example 7.9 of \cite{FK_glue}, we found 
\[ p(n,o,o) \approx \frac{C}{n \log^2(1+n)}.\]

Note that $(\Gamma_h, \mathcal{K}_h, \pi_h)$ \emph{does not} satisfy the volume growth condition (\ref{minvol_faster2}), since in the $\Z^2$ page, volume grows like a power of $2$ times log squared. Therefore, we cannot directly apply Theorem \ref{main_general_thm} or Corollary \ref{graphs_finite_case}. However, we can apply the ideas of the proof of Theorem \ref{main_general_thm} to the abstract gluing formula Theorem \ref{gluing_thm} without difficulty. 

Using the estimates from Section \ref{term_estimates} and the central bound above, we have: 
\begin{align*}
    &p_h(n,o,o) \approx \frac{C}{n \log^2(1+n)} \\
    &\sum_{l=0}^n p(l,o,o) \approx C \\
    &\psi'_{\partial \Gamma_1, h}(n,x,o) \approx \frac{C}{n^{D/2}} \exp\Big(-\frac{|x|^2}{cn}\Big)\\
    &\psi'_{\partial \Gamma_2, h}(n,x,o) \approx \frac{C}{n \log^2(1+|x|+\sqrt{n})} \exp\Big(-\frac{|x|^2}{cn}\Big)\\
    &\psi_{\partial \Gamma_1, h}(n,x,o) \approx \frac{C}{|x|^{D-2}} \exp\Big(-\frac{|x|^2}{cn}\Big)\\
    &\psi_{\partial \Gamma_2, h}(n,x,o) \approx \frac{C}{\log^2(1+|x|}\exp\Big(-\frac{|x|^2}{cn}\Big) + \Big( \frac{C_1}{\log(1+|x|)} - \frac{C_2}{\log(1+n)} \Big)_+\\
    &h(x) \approx 1 \ \forall x \in \Z^D; \qquad h(x) \approx \log(1+|x|) \ \forall x \in \Z^2
\end{align*}
In the second term of $\psi_{\partial_2}(n,x,o)$, one can multiply by $\exp(-|x|^2/cn)$ or not; since this term came from the portion of the sum where $|x|^2 \leq n,$ the exponential is large. 

Substituting these estimates into Theorem \ref{gluing_thm} gives the following heat kernel estimates based on the location of $x$ and $y$:

\begin{enumerate}[(1)]
    \item The central estimate, mentioned above, is 
    \[ p(n,o,o) \approx \frac{C}{n\log^2(1+n)}.\]
    \item If $x, y \in \Z^2$,
    \[p(n,x,y) \approx \frac{C\log(1+|x|) \log(1+|y|)}{n\log(1+|x|+\sqrt{n}) \log(1+|y|+\sqrt{n})} \exp\Big(-\frac{d^2(x,y)}{cn}\Big).\]
    Here the denominator is written as $\sqrt{V_{i_x, h}(x, \sqrt{n}) V_{i_y,h}(x, \sqrt{n})}$; however, since $(\Z^2, \pi_h, \mu_h)$ is a Harnack graph with doubling volume function, it is possible to replace the denominator with either $V_{i_x}(x, \sqrt{n}) \approx n \log^2(1+|x|+\sqrt{n})$ or $V_{i_y}(y, \sqrt{n}) \approx n \log^2(1+|y|+\sqrt{n})$, up to changes of the (unwritten) constants inside and outside the exponential.
    \item If $x,y \in \Z^D,$
    \begin{align*}
    p(n,x,y) &\approx \frac{C}{n^{D/2}}\exp\Big(-\frac{d^2(x,y)}{cn}\Big) \\&+ \frac{C}{n\log^2(1+n)|x|^{D-2}|y|^{D-2}}\exp\Big(-\frac{(|x|^2+|y|^2)}{cn}\Big).
    \end{align*}
    \item If $x \in \Z^2, y \in \Z^D,$
    \begin{align*}
    p(n,x,y) \approx C\bigg[&\frac{\log(1+|x|)}{n\log^2(1+n) |y|^{D-2}} \\&+ \frac{1}{n^{D/2}}\bigg( \frac{1}{\log(1+|x|)} + \Big( 1-\frac{\log(1+|x|)}{\log(1+n)}\Big)_+ \bigg)\bigg] \exp\Big(-\frac{(|x|^2+|y|^2}{cn}\Big)
    \end{align*}
    \item Suppose $x$ is in one page while $y=o.$ If $x \in \Z^2$ and $y=o$, then (2) and (4) give
    \begin{align*}
    p(n,x,o) &\approx \frac{C\log(1+|x|)}{n\log^2(1+|x|+\sqrt{n})}\exp\Big(-\frac{|x|^2}{cn}\Big)\\
    &\approx \frac{C\log(1+|x|)}{n\log^2(1+n)} \exp\Big(-\frac{|x|^2}{cn}\Big).
    \end{align*}
    While it may not be immediately obvious that these two estimates are the same, the exponential makes up for the difference; recall the equivalent ways to write (2). 
    
    If $x \in \Z^D$ and $y=o,$ then (3) and (4) give
    \[ p(n,x,o) \approx C\bigg[\frac{1}{n\log^2(1+n) |x|^{D-2}} + \frac{1}{n^{D/2}} \bigg]\exp\Big(-\frac{|x|^2}{cn}\Big).\]
\end{enumerate}
Again, these estimates match the continuous case results of Grigor'yan and Saloff-Coste (Example 6.14 of \cite{lsc_ag_ends}) and of Ooi (Theorem 1.6 of \cite{ooi_bmvd}). 
    
\end{example}

\begin{remark}

As we have seen in the previous two examples, our technique can handle graphs with $S$-recurrent pages, so long as we are capable of constructing a global harmonic function $h$ on the glued graph such that all pages of the $h$-transform graph are uniformly $S$-transient. The two particular examples we considered were simple enough that this was fairly straightforward. However, in general, showing the existence and desirable properties of such an $h$ is a harder question. This is especially so because as we showed in our previous paper \cite{ed_lsc_trans}, there exist graphs that are $S$-transient but not uniformly so. 

Consider for instance the case of a graph with one page that is $S$-transient but not uniformly so and all other pages uniformly $S$-transient. This graph should be transient in the classical sense, and therefore all harmonic functions are roughly constant. Therefore it is impossible to use an $h$-transform to get a graph with all \emph{uniformly} $S$-transient pages. 

One then imagines the difficulty of showing in a more general situation that the $h$-transform of an $S$-recurrent page is \emph{uniformly} $S$-transient. 

To get estimates for general graphs with pages that may not be uniformly $S$-transient, a much better understanding of both the possible harmonic functions $h$ and $S$-transience (or perhaps a different transience-type hypotheses) is needed. 
\end{remark}


\begin{thebibliography}{10}

\bibitem{Barlow_graphs}
Martin~T. Barlow.
\newblock {\em Random walks and heat kernels on graphs}, volume 438 of {\em
  London Mathematical Society Lecture Note Series}.
\newblock Cambridge University Press, Cambridge, 2017.

\bibitem{chen_lou_BMvary}
Zhen-Qing Chen and Shuwen Lou.
\newblock Brownian motion on some spaces with varying dimension.
\newblock {\em Ann. Probab.}, 47(1):213--269, 2019.

\bibitem{TC_LSC_IsoInfini}
Thierry Coulhon and Laurent Saloff-Coste.
\newblock Vari\'et\'es riemanniennes isom\'etriques \`a{} l'infini.
\newblock {\em Rev. Mat. Iberoamericana}, 11(3):687--726, 1995.

\bibitem{Emily_thesis}
Emily Dautenhahn.
\newblock {\em Heat kernel estimates on glued spaces}.
\newblock PhD thesis, Cornell University, 2024.

\bibitem{ed_lsc_trans}
Emily Dautenhahn and Laurent Saloff-Coste.
\newblock Hitting probabilities and uniformly {$S$}-transient subgraphs.
\newblock {\em Electron. J. Probab.}, 29:Paper No. 80, 33, 2024.

\bibitem{ed_lsc_mixbdry}
Emily Dautenhahn and Laurent Saloff-Coste.
\newblock Heat kernel estimates on manifolds with ends with mixed boundary
  condition.
\newblock {\em Comm. Anal. Geom.}, 33(3):651--699, 2025.

\bibitem{FK_glue}
Emily Dautenhahn and Laurent Saloff-Coste.
\newblock Faber-{K}rahn inequality and heat kernel estimates on glued graphs.
\newblock {\em Electron. J. Probab.}, 31, 2026.

\bibitem{Delmotte_PHI}
Thierry Delmotte.
\newblock Parabolic {H}arnack inequality and estimates of {M}arkov chains on
  graphs.
\newblock {\em Rev. Mat. Iberoamericana}, 15(1):181--232, 1999.

\bibitem{pd_khe_lsc_analytic}
Persi Diaconis, Kelsey Houston-Edwards, and Laurent Saloff-Coste.
\newblock Analytic-geometric methods for finite {M}arkov chains with
  applications to quasi-stationarity.
\newblock {\em ALEA Lat. Am. J. Probab. Math. Stat.}, 17(2):901--991, 2020.

\bibitem{anthony_bdryHarnack3g}
Anthony Graves-McCleary and Laurent Saloff-Coste.
\newblock The boundary {H}arnack principle and the {$3G$} principle in
  fractal-type spaces.
\newblock {\em Math. Nachr.}, 298(11):3554--3575, 2025.

\bibitem{Gri}
Alexander Grigor'yan.
\newblock The heat equation on noncompact {R}iemannian manifolds.
\newblock {\em Mat. Sb.}, 182(1):55--87, 1991.

\bibitem{parabolic_ends}
Alexander Grigor'yan, Satoshi Ishiwata, and Laurent Saloff-Coste.
\newblock Heat kernel estimates on connected sums of parabolic manifolds.
\newblock {\em J. Math. Pures Appl. (9)}, 113:155--194, 2018.

\bibitem{G_I_LSC_survey}
Alexander Grigor'yan, Satoshi Ishiwata, and Laurent Saloff-Coste.
\newblock Geometric analysis on manifolds with ends.
\newblock In {\em Analysis and partial differential equations on manifolds,
  fractals and graphs}, volume~3 of {\em Adv. Anal. Geom.}, pages 325--343. De
  Gruyter, Berlin, [2021] \copyright 2021.

\bibitem{Poincare_const}
Alexander Grigor'yan, Satoshi Ishiwata, and Laurent Saloff-Coste.
\newblock Poincar\'{e} constant on manifolds with ends.
\newblock {\em Proc. Lond. Math. Soc. (3)}, 126(6):1961--2012, 2023.

\bibitem{GS1}
Alexander Grigor'yan and Laurent Saloff-Coste.
\newblock Heat kernel on connected sums of {R}iemannian manifolds.
\newblock {\em Math. Res. Lett.}, 6(3-4):307--321, 1999.

\bibitem{ag_lsc_extcptset}
Alexander Grigor'yan and Laurent Saloff-Coste.
\newblock Dirichlet heat kernel in the exterior of a compact set.
\newblock {\em Comm. Pure Appl. Math.}, 55(1):93--133, 2002.

\bibitem{ag_lsc_hittingprob}
Alexander Grigor'yan and Laurent Saloff-Coste.
\newblock Hitting probabilities for {B}rownian motion on {R}iemannian
  manifolds.
\newblock {\em J. Math. Pures Appl. (9)}, 81(2):115--142, 2002.

\bibitem{ag_lsc_harnackstability}
Alexander Grigor'yan and Laurent Saloff-Coste.
\newblock Stability results for {H}arnack inequalities.
\newblock {\em Ann. Inst. Fourier (Grenoble)}, 55(3):825--890, 2005.

\bibitem{lsc_ag_ends}
Alexander Grigor'yan and Laurent Saloff-Coste.
\newblock Heat kernel on manifolds with ends.
\newblock {\em Ann. Inst. Fourier (Grenoble)}, 59(5):1917--1997, 2009.

\bibitem{ag_lsc_FKsurgery}
Alexander Grigor'yan and Laurent Saloff-Coste.
\newblock Surgery of the {F}aber-{K}rahn inequality and applications to heat
  kernel bounds.
\newblock {\em Nonlinear Anal.}, 131:243--272, 2016.

\bibitem{ag_ishiwata}
Alexander Grigor?yan and Satoshi Ishiwata.
\newblock Heat kernel estimates on a connected sum of two copies of
  $\mathbb{R}^n$ along a surface of revolution.
\newblock {\em Global and Stochastic Analysis}, 2(1):29--65, 2012.

\bibitem{mathav_jaschek}
Tim Jaschek and Mathav Murugan.
\newblock Geometric implications of fast volume growth and capacity estimates.
\newblock In {\em Analysis and partial differential equations on manifolds,
  fractals and graphs}, volume~3 of {\em Adv. Anal. Geom.}, pages 183--199. De
  Gruyter, Berlin, [2021] \copyright 2021.

\bibitem{lawler_limic}
Gregory~F. Lawler and Vlada Limic.
\newblock {\em Random walk: a modern introduction}, volume 123 of {\em
  Cambridge Studies in Advanced Mathematics}.
\newblock Cambridge University Press, Cambridge, 2010.

\bibitem{lou_li_distorted}
Liping Li and Shuwen Lou.
\newblock Distorted {B}rownian motions on space with varying dimension.
\newblock {\em Electron. J. Probab.}, 27:Paper No. 72, 32, 2022.

\bibitem{lou_BMVD_drift}
Shuwen Lou.
\newblock Brownian motion with drift on spaces with varying dimension.
\newblock {\em Stochastic Process. Appl.}, 129(6):2086--2129, 2019.

\bibitem{lou_distorted}
Shuwen Lou.
\newblock Explicit heat kernels of a model of distorted {B}rownian motion on
  spaces with varying dimension.
\newblock {\em Illinois J. Math.}, 65(2):287--312, 2021.

\bibitem{ooi_bmvd}
Takumu Ooi.
\newblock Heat kernel estimates on spaces with varying dimension.
\newblock {\em Tohoku Math. J. (2)}, 74(2):165--194, 2022.

\bibitem{PSHDuke}
Laurent Saloff-Coste.
\newblock A note on {P}oincar\'{e}, {S}obolev, and {H}arnack inequalities.
\newblock {\em Internat. Math. Res. Notices}, (2):27--38, 1992.

\bibitem{Spitzer}
Frank Spitzer.
\newblock {\em Principles of random walk}, volume~34.
\newblock Springer Science \& Business Media, 2013.

\end{thebibliography}

\end{document}